\documentclass[reqno]{amsart}
\RequirePackage[l2tabu, orthodox]{nag}

\usepackage{amssymb}
\usepackage[style=numeric]{biblatex}
\usepackage{booktabs}
\usepackage{color}
\usepackage[pagewise]{lineno}
\usepackage{mathtools}
\usepackage{microtype}
\usepackage{svg}
\usepackage{tcolorbox}
\usepackage{tikz-cd}
\tcbuselibrary{theorems}

\usepackage{hyperref}
\usepackage{mathrsfs}

\theoremstyle{plain}
\newtheorem{theorem}{Theorem}[section]

\newtheorem{proposition}[theorem]{Proposition}
\newtheorem{lemma}[theorem]{Lemma}

\theoremstyle{definition}

\theoremstyle{remark}
\newtheorem{remark}[theorem]{Remark}

\newcommand{\IGNORE}[1]{}

\newcommand{\C}{\mathbb{C}}

\newcommand{\R}{\mathbb{R}}

\newcommand{\BA}{\boldsymbol{A}}
\newcommand{\BJ}{\boldsymbol{j}}
\newcommand{\BP}{\boldsymbol{p}}

\newcommand{\BU}{\boldsymbol{u}}
\newcommand{\BV}{\boldsymbol{v}}
\newcommand{\BW}{\boldsymbol{w}}
\newcommand{\BX}{\boldsymbol{x}}

\newcommand{\EV}{\mathrm{even}}
\newcommand{\CL}[1]{\mathrm{Cl}_{#1}}
\newcommand{\FA}{\mathfrak{A}}
\newcommand{\FB}{\mathfrak{b}}

\newcommand{\FE}{\mathfrak{e}}

\newcommand{\FJ}{\mathfrak{j}}
\newcommand{\FL}{\mathfrak{l}}
\newcommand{\FN}{\mathfrak{n}}
\newcommand{\FP}{\mathfrak{p}}
\newcommand{\FR}{\mathfrak{r}}
\newcommand{\FQ}{\mathfrak{q}}

\newcommand{\FU}{\mathfrak{u}}

\newcommand{\GS}{\geqslant}
\newcommand{\IM}{\mathrm{Im}}

\newcommand{\LS}{\leqslant}

\newcommand{\OP}{\overline{P}}

\newcommand{\RE}{\mathrm{Re}}

\newcommand{\UO}{\mathbf{U}}
\newcommand{\VE}{\mathrm{v}}

\newcommand{\BVV}{\boldsymbol{V}}

\newcommand{\ONE}{\boldsymbol{1}}
\newcommand{\RHO}{\varrho}
\newcommand{\TRACE}{\mathrm{tr}}

\newcommand{\OPHI}{\overline{\phi}}
\newcommand{\WPHI}{\widehat{\phi}}

\newcommand{\BGAMMA}{\boldsymbol{\Gamma}}
\newcommand{\BNABLA}{\boldsymbol{\nabla}}

\newcommand{\BTHETA}{\boldsymbol{\theta}}

\numberwithin{equation}{section}

\title{A Hydrodynamics Formulation for a Nonlinear Dirac Equation}

\author{Joan Morrill i Gavarró}
\author{Michael Westdickenberg}

\address{%
    Joan Morrill i Gavarró,
    Lehrstuhl f\"{u}r Mathematik (Analysis),
    RWTH Aachen University,
    Im Süsterfeld 2,
    D-52072 Aachen,
    Germany}
\email{morrill@eddy.rwth-aachen.de}

\address{%
    Michael Westdickenberg,
    Lehrstuhl f\"{u}r Mathematik (Analysis),
    RWTH Aachen University,
    Im Süssterfeld 2,
    D-52072 Aachen,
    Germany}
\email{mwest@instmath.rwth-aachen.de}

\date{\today}

\subjclass[2010]{%
	15A66, 
	35Q41, 
	81R25} 

\keywords{Nonlinear Dirac equation, hydrodynamics formulation, Clifford algebra}

\thanks{The authors thank Deutsche Forschungsgemeinschaft/German Research Foundation (DFG) for financial support through 320021702 / GRK2326.}

\begin{document}

\begin{abstract}
We derive a hydrodynamics formulation for a modified Dirac equation with a nonlinear mass term that preserves the homogeneity of the original Dirac equation. The nonlinear Dirac equation admits a symmetric split into the left and right-handed spinor components. It is formulated using Clifford algebra tools. We prove global existence for a regularized equation.
\end{abstract}

\maketitle



\section{Introduction}

It is well-known that the Schrödinger equation of quantum mechanics
\[
	i\hbar \frac{\partial}{\partial t} \psi(t,\BX)
		= \left( \frac{-\hbar^2}{2m} \Delta + V(t,\BX) \right) \psi(t,\BX)
\]
can be rewritten in the form of a system of conservation laws in terms for hydrodynamics variables: the density $\RHO$ and an (essentially) irrotational velocity $\BU$. To this end, one starts with the polar decomposition of the $\C$-valued wave function
\[
	\psi(t,\BX) \eqcolon \sqrt{\frac{1}{m} \RHO(t,\BX)} e^{\frac{i}{\hbar} S(t,\BX)},
\]
which defines $\RHO$ and a real-valued phase function $S$. The velocity is then given as
\[
	\BU(t,\BX) \coloneq \frac{1}{m} \nabla S(t,\BX),
\]
and $(\RHO, \BU)$ together satisfy the compressible fluid equations
\begin{equation}
\begin{alignedat}{2}
	\partial_t \RHO + \nabla\cdot(\RHO\BU) &= 0
		&\quad& \text{continuity equation,}
\\
	\partial_t \BU + \BU\cdot\nabla \BU + \frac{1}{m} \nabla(V+Q) &= 0
		&& \text{velocity equation.}
\end{alignedat}
\label{E:JJH}
\end{equation}
The quantum effects are captured in the non-classical field
\begin{equation}
	Q \coloneq -\frac{\hbar^2}{2m} \frac{\Delta\sqrt{\RHO}}{\sqrt{\RHO}},
\label{E:BOHM}
\end{equation}
called quantum potential. It is the only term in \eqref{E:JJH} that is multiplied by the Planck constant $\hbar$. The velocity equation can alternatively be replaced by a Hamilton-Jacobi type equation for the phase field $S$. The outlined procedure is known as the Madelung transform or hydrodynamics formulation; see \cites{Madelung1926, Madelung1927}.

\medskip

The Schrödinger equation is understood as the non-relativistic approximation of the more fundamental Dirac equation, which we will discuss in more detail in Section \ref{S:IDE}. There has been substantial work trying to establish a hydrodynamics formulation of the Dirac equation, most notably by Takabayasi \cite{Takabayasi1957}. The resulting model is rather complicated and suffers from the occurrence of a mysterious scalar field, called the Yvon-Takabayasi angle, which spoils the symmetry and nonnegativity of mass and energy. This seems to be a consequence of the linearity of the equation.

\medskip

In this paper, we propose a different hydrodynamics formulation for the Dirac equation. Our work is based on two important ingredients. First, we use the \emph{space algebra} $\CL{3}$ instead of the more commonly used $\C^4$-valued spinor functions. Clifford algebra tools have been utilized before for studying the Dirac equation, most notably by Hestenes; see \cite{Hestenes1973} and the references therein. The relevant structure suggested by the classical form of the Dirac equation is that of (an even subalgebra of) the \emph{spacetime algebra}. We prefer the more compact $\CL{3}$ representation, which simplifies computations significantly. We follow the exposition and notation of Baylis \cite{Baylis1996}.

Second, we consider a \emph{nonlinear} variant of the Dirac equation. The model we use is different from, say, the cubic Soler model or the Thirring model, which have attracted considerable interest. Instead we use a model proposed by Daviau (see \cite{DaviauBertrandSocrounGirardot2022} and the references therein), which in a sense just adds the minimal amount of nonlinearity needed to achieve an additional $\mathrm{U}(1)$ symmetry, while keeping the first-order homogeneity of the equation. As discussed in \cite{DaviauBertrandSocrounGirardot2022}, the model correctly predicts the electronic energy levels in a hydrogen atom -- the fundamental problem that motivated Dirac. It admits a very natural splitting of the spinor into left and right-handed components, which will be crucial for our approach.

Indeed the hydrodynamics formulation we propose in this paper splits naturally into two parts for the left and right waves, which are only weakly coupled through the Dirac current. In the words of De Broglie, we like to think of the Dirac current as the pilot wave that guides the evolution of the particles. The resulting model admits two sets of congruences of flowlines, one each for the left and right waves, that wind around the flowlines determined by the Dirac current. Such a model has been considered before by Hestenes \cite{Hestenes2020a, Hestenes2020b} who connects the helical motion to spin and to Schrödinger's prediction of a rapid oscillatory motion of elementary particles, for which he coined the term Zitterbewegung.


\section{Space algebra}
\label{S:SA}

The space algebra $\CL{3}$ (also called the geometric algebra) is the Clifford algebra of the three-dimensional Euclidean space $\R^3$. We will use the following correspondence: Letters with arrows mark elements in $\R^3$, with the same letter in bold representing an element in $\CL{3}$ with the same coefficients. Thus
\begin{equation}
    \vec{x} = \begin{pmatrix} x^1 \\ x^2 \\ x^3 \end{pmatrix} \in \R^3
    \quad\Longleftrightarrow\quad
    \BX = x^k \FE_k \in \CL{3}.
\label{E:EMB}
\end{equation}
Here $\FE_k$ are suitable basis vectors in $\CL{3}$ specified below. Throughout, we employ the Einstein summation convention, which means summation over $1\ldots 3$ in case of repeated Latin indices. The Clifford algebra $\CL{3}$ augments the vector space structure of $\R^3$ with an associative bilinear product with the property that
\begin{equation}
    \BX^2 = \BX \BX = \vec{x}\cdot\vec{x}
    \quad\text{for all $\vec{x} \in \R^3$,}
\label{E:FUNDA}
\end{equation}
where $\cdot$ marks the Euclidean inner product
\begin{equation}
	\vec{x}\cdot\vec{y} \coloneq \sum_{k=1}^3 x^k y^k
	\quad\text{for all $\vec{x}, \vec{y} \in \R^3$,}
\label{E:CDOT}
\end{equation}
with $\|\vec{x}\| \coloneq \sqrt{\vec{x}\cdot\vec{x}}$ the induced norm. By polarization, \eqref{E:FUNDA} is equivalent to
\[
	\BU \BV + \BV \BU = 2 \vec{u}\cdot\vec{v}
	\quad\text{for all $\vec{u}, \vec{v} \in \R^3$,}
\]
which for the basis vectors $\FE_k$ translates into the structure equations
\begin{equation}
    \FE_k \FE_l + \FE_l \FE_k = 2\delta_{kl}
	\quad\text{for $k,l = 1 \ldots 3$,}
\label{E:STRUCTURE}
\end{equation}
with $\delta_{kl}$ the Kronecker delta. In particular, we need $\FE_k^2 = 1$ and $\FE_k \FE_l + \FE_l \FE_k = 0$ if $k\neq l$. Any product of basis vectors $\FE_k$ can be reduced to one from the family
\begin{equation}
    \{ \FE_0 \coloneq 1, \FE_k, \FE_k \FE_j, \FE_1 \FE_2 \FE_3 \}
    \quad\text{with $1 \LS k < j \LS 3$,}
\label{E:BASIS}
\end{equation}
up to a sign. As a consequence, $\CL{3}$ is $8$-dimensional as a vector space over $\R$. This is the same dimension as for Dirac spinors, which classically are given as $\C^4$-valued fields. Note that $\FE_0$ commutes with $\FE_k$, $k = 1\ldots 3$, and $\FE_0^2 = 1$. We will typically not explicitly mark $\FE_0$ when the meaning is clear from the context.

Using \eqref{E:STRUCTURE} one can check that $\FE_1 \FE_2 \FE_3$ squares to $-1$ and commutes with all basis vectors in \eqref{E:BASIS}. In the following, we will therefore identify $\FE_1 \FE_2 \FE_3$ with the imaginary unit $i$. It is convenient to consider $\CL{3}$ as a vector space over $\C$, for which the vectors $\FE_\mu$ with $\mu = 0\ldots 3$, form a basis. We will use the following correspondence: Letters in fraktur mark elements in $\CL{3}$, which we will call paravectors; see \cite{Baylis1996}. The same letter with index $0$ represents the scalar part of this paravector, the same letter in bold its spatial part, which can be expanded as in \eqref{E:EMB}. Thus
\[
    (p^0, \vec{p}) \in \C \times \C^3
    \quad\Longleftrightarrow\quad
    \FP \coloneq p^0 + \BP \coloneq p^\mu \FE_\mu \in \CL{3}.
\]
We extend \eqref{E:FUNDA}, \eqref{E:CDOT} to $\vec{x},\vec{y} \in \C^3$. Note, however, that \eqref{E:CDOT} is not an inner product because $\vec{x}\cdot\vec{x}$ may not be nonnegative (no complex conjugation in \eqref{E:CDOT}).

Vectors in $\CL{3}$ with \emph{real} coefficients are called spacetime vectors. They will play an important role in the following. Examples of spacetime vectors include
\begin{itemize}
\item the $4$-velocity $\FU = u^0 + \BU \eqcolon \gamma (1 + \BV)$,
\item the chiral current $\FJ = j^0 + \BJ$ with density $\RHO \coloneq j^0$,
\item the energy-momentum $\FP = p^0 + \BP$ with energy $E \coloneq p^0$,
\item the electromagnetic potential $\FA = A^0 + \BA$.
\end{itemize}


\subsection{Involutions}
\label{S:INVO}

Because of \eqref{E:ANTICO}, a general element $\FP \in \CL{3}$ is of the form
\begin{equation}
    \FP = a + \BU + i \BV + i b
    \quad\text{with $(a,\vec{u}), (b,\vec{v}) \in \R\times\R^3$.}
\label{E:PEQ}
\end{equation}
Thus $\FP$ is a sum of a spacetime vector $a+\BU$ and $i$ times another spacetime vector $b+\BV$. We can rewrite this, combining real and imaginary numbers, as
\begin{equation}
    \FP = c + \BW
    \quad\text{with $c \coloneq a+ib$ and $\BW \coloneq \BU+i\BV$.}
\label{E:PEGC}
\end{equation}

\medskip

(1) We define the \emph{complex conjugation} of \eqref{E:PEQ} as
\[
    \FP^\dagger \coloneq a + \BU - i\BV - ib.
\]
This operation changes the sign of parts multiplied by $i$, but leaves the basis vectors $\FE_\mu$, $\mu=0\ldots 3$, unchanged. It is an involution (applying the operation twice returns $\FP$) and can be used to isolate the real and imaginary parts of $\FP$, defined as
\begin{equation}
    \RE(\FP) \coloneq \frac{1}{2} (\FP+\FP^\dagger)
    \quad\text{and\quad}
    \IM(\FP) \coloneq \frac{1}{2i} (\FP-\FP^\dagger).
\label{E:REIM}
\end{equation}
Spacetime vectors $\FP = a+\BU$ with $a\in \R$, $\vec{u} \in \R^3$  can therefore be characterized as those elements $\FP \in \CL{3}$ that are real, meaning that $\FP = \FP_\RE$.


\medskip

(2) We define the \emph{grade automorphism} of \eqref{E:PEQ} as
\[
	\widehat{\FP} \coloneq a - \BU + i\BV - ib.
\]
It is again an involution. Basis vectors are changed according to
\begin{equation}
	\widehat{\FE_0} = \FE_0
	\quad\text{and}\quad
	\text{$\widehat{\FE_k} = -\FE_k$ for $k=1\ldots 3$.}
\label{E:IDD}
\end{equation}
The grade automorphism can be used to isolate even and odd parts, defined as
\[
    \FP_\EV \coloneq \frac{1}{2} (\FP+\widehat{\FP})
    \quad\text{and\quad}
    \FP_\mathrm{odd} \coloneq \frac{1}{2} (\FP-\widehat{\FP}).
\]
The subset of even elements (thus $\widehat{\FP} = \FP$) is a subalgebra of $\CL{3}$. It is four-dimensional as a vector space over $\R$ and generated by the basis vectors $\FE_0=1$ and $i\FE_k$, $k=1\ldots 3$. Since $(i\FE_k)^2 = -1$, the even subalgebra is isomorphic to the quaternions.


\medskip

(3) We define the \emph{spatial reversal} of \eqref{E:PEQ}/\eqref{E:PEGC} as
\begin{equation}
	\overline{\FP} \coloneq a - \BU - i\BV + ib = c - \BW.
\label{E:SPATR}
\end{equation}
This operation only reverses the sign of the vector part, but leaves the scalar part unchanged. It is again an involution. Basis vectors change according to
\begin{equation}
	\overline{\FE_0} = \FE_0
	\quad\text{and}\quad
	\text{$\overline{\FE_k} = -\FE_k$ for $k=1 \ldots 3$.}
\label{E:EBA}
\end{equation}
The spatial reversal can be used to compute the scalar and vector parts of $\FP$:
\begin{equation}
	\FP_0 \coloneq \frac{1}{2} (\FP+\overline{\FP})
    \quad\text{and\quad}
    \FP_\VE \coloneq \frac{1}{2} (\FP-\overline{\FP}).
\label{E:SCALAR}
\end{equation}
Note that the $\FE^0$-coefficient of $\FP$ is the scalar part; therefore the notation is consistent. We observe further that for any $\FP \in \CL{3}$ as in \eqref{E:PEQ} there holds
\[
	 \big( \RE(\FP) \big)_0 = (a+\BV)_0 = a = \RE(a+ib) = \RE( \FP_0 );
\]
the operations commute. One can check that $\widehat{\FP} = \overline{\FP}^\dagger$, so that only two of the three involutions are actually independent. For all $\FP_1, \FP_2$ as in \eqref{E:PEGC}/\eqref{E:PEQ}, we have
\begin{equation}
	\FP_1\overline{\FP_2} + \FP_2\overline{\FP_1}
		= 2 c_1 c_2 - (\BW_1\BW_2 + \BW_2\BW_1)
		= \overline{\FP_1}\FP_2 + \overline{\FP_2}\FP_1,
\label{E:PP3}
\end{equation}
which is \emph{always a (complex) scalar} because $\BW_1\BW_2 + \BW_2\BW_1 = 2 \vec{w}_1 \cdot \vec{w}_2$.
(In contrast, using complex conjugation or grade automorphism in an analogous construction provides a scalar product for the real and even \emph{subalgebras}, respectively, but the product of two general elements in $\CL{3}$ may have scalar \emph{and vector} parts.)

\medskip

The three involutions act differently on products. We have
\begin{equation}
    (\FP\FQ)^\dagger = \FQ^\dagger \FP^\dagger
    \quad\text{and}\quad
    \overline{\FP\FQ} = \overline{\FQ} \, \overline{\FP},
    \quad\text{whereas}\quad
    \widehat{\FP\FQ} = \widehat{\FP} \, \widehat{\FQ}.
\label{E:ORDER}
\end{equation}
This can again be proved using the decomposition of a paravector into its components as in \eqref{E:PEQ} or \eqref{E:PEGC}. One only needs to take into consideration that the product of two spatial vectors has both a scalar and a vector part, which transform differently. Consider for example two paravectors $\FP_k = c_k+\BW_k$, $k=1,2$, as in \eqref{E:PEGC}. Then
\begin{align*}
	\FP_1\FP_2
		& = (c_1+\BW_1)(c_2+\BW_2)
\\
		& = (c_1c_2 + \vec{w}_1 \cdot \vec{w}_2)
		+ \Big( c_1\BW_2 + c_2\BW_1 + {w_1}^k {w_2}^l \, i{\varepsilon_{kl}}^m \FE_m \Big),
\end{align*}
where we have used \eqref{E:EMB} and \eqref{E:ANTICO}. Its spatial reversal is
\begin{equation}
	\overline{\FP_1\FP_2}
		= (c_1c_2 + \vec{w}_1 \cdot \vec{w}_2)
		- \Big( c_1\BW_2 + c_2\BW_1 + {w_1}^k {w_2}^l \, i{\varepsilon_{kl}}^m \FE_m \Big),
\label{E:HELP1}
\end{equation}
because this operation changes the sign of vector parts. Similarly, we compute
\begin{align*}
	\overline{\FP_2} \, \overline{\FP_1}
		& = (c_2-\BW_2)(c_1-\BW_1)
\\
		& = (c_1c_2 + \vec{w}_1\cdot\vec{w}_2)
		+ \Big( -c_1\BW_2 - c_2\BW_1 + {w_2}^l {w_1}^k \, i{\varepsilon_{lk}}^m \FE_m \Big).
\end{align*}
This coincides with \eqref{E:HELP1} because ${\varepsilon_{kl}}^m = -{\varepsilon_{lk}}^m$; see \eqref{E:LC}. We also have
\[
	(\FP_1\FP_2)_0
		= \frac{1}{2} (\FP_1\FP_2 + \overline{\FP_1\FP_2})
		= c_1 c_2 + \vec{w}_1 \cdot \vec{w}_2,
\]
which is clearly symmetric in the two arguments. It follows that
\begin{equation}
	(\FP\FQ)_0 = (\FQ\FP)_0
	\quad\text{for all $\FP,\FQ \in \CL{3}$,}
\label{E:ROTAT}
\end{equation}
and arguments can be rotated cyclically without changing the scalar value.


\subsection{Inverse elements}
\label{SS:IE}

One big advantage of using a Clifford algebra instead of $\C^4$ when discussing the Dirac equation, comes from the fact that elements in $\CL{3}$ may have multiplicative inverses. We say that $\FP \in \CL{3}$ has a multiplicative inverse if there exists $\FQ \in \CL{3}$ with the property that $\FP\FQ$ is a \emph{nonzero complex scalar}. Then we can define the right inverse $\FR \coloneq (\FP\FQ)^{-1} \FQ$ so that
\[
    \FP\FR = \FP \Big( (\FP\FQ)^{-1}\FQ \Big) = (\FP\FQ)^{-1} (\FP\FQ) = 1.
\]
The left inverse is defined analogously as $\FL \coloneq (\FQ\FP)^{-1} \FQ$. In order to find an element $\FQ$ as above, we make use of the spatial reversal because \eqref{E:PP3} shows that $\FP\overline{\FP}$ is always a scalar. Elements $\FP \in \CL{3}$ with non-vanishing $\FP\overline{\FP}$ are thus invertible with
\[
    \FP^{-1} = (\FP\overline{\FP})^{-1} \overline{\FP},
\]
which functions both as a left and right inverse because $\FP\overline{\FP} = \overline{\FP}\FP$; see \eqref{E:PP3}.


\subsection{Scalar product}

For a paravector $\FP = a+\BU$ with $(a,\vec{u}) \in \R \times \R^3$, we obtain from \eqref{E:PP3} the Minkowski spacetime length $\FP\overline{\FP} = a^2 - \|\vec{u}\|^2$. We call $\FP$
\begin{alignat*}{2}
    \text{timelike}
        &\quad& \text{if $\FP\overline{\FP} > 0$,} \\
    \text{lightlike or null}
        && \text{if $\FP\overline{\FP} = 0$,} \\
    \text{spacelike}
        && \text{if $\FP\overline{\FP} < 0$.}
\end{alignat*}
More generally, we introduce the scalar product
\begin{equation}
	\langle \FP,\FQ \rangle
		\coloneq (\FP\overline{\FQ})_0
		= \frac{1}{2} (\FP\overline{\FQ} + \FQ\overline{\FP})
	\quad\text{for all $\FP,\FQ \in \CL{3}$,}
\label{E:SCPR}
\end{equation}
using \eqref{E:SCALAR} and \eqref{E:ORDER}. Notice that \eqref{E:SCPR} is indeed always scalar because of \eqref{E:PP3}. It is symmetric in $\FP,\FQ$ and bilinear in both arguments. Using \eqref{E:EBA}, we find
\begin{equation}
	\langle \FE_\mu, \FE_\nu \rangle
		= \frac{1}{2} (\FE_\mu \overline{\FE_\nu} + \FE_\nu \overline{\FE_\mu})
		= \eta_{\mu\nu}
	\quad\text{for $\mu,\nu = 0\ldots 3$,}
\label{E:MINKO}
\end{equation}
with metric tensor $\eta_{00} = 1$, $\eta_{11} = \eta_{22} = \eta_{33} = -1$, and $\eta_{\mu\nu} = 0$ otherwise.

We define dual basis vectors
\begin{equation}
	\FE^0 \coloneq \FE_0
	\quad\text{and}\quad
	\text{$\FE^k \coloneq -\FE_k$ for $k=1\ldots 3$.}
\label{E:DUALB}
\end{equation}
Then the involutions of basis vectors can be expressed as
\begin{equation}
	\text{$\overline{\FE_\mu} = \widehat{\FE_\mu} = \FE^\mu$ for $\mu=0\ldots 3$.}
\label{E:RAISER}
\end{equation}
Then \eqref{E:MINKO} implies the orthogonality relation
\begin{equation}
	\langle \FE^\mu, \FE_\nu \rangle = {\delta^\mu}_\nu
	\quad\text{for $\mu,\nu = 0\ldots 3$.}
\label{E:ORTHO}
\end{equation}
The $\FE^\mu$ provide another orthogonal basis of $\CL{3}$. The metric tensor and $\eta^{\mu\nu} \coloneq \eta_{\mu\nu}$ can be used to lower/raise indices, thus to switch between the $\FE_\mu$ and $\FE^\mu$ bases.

\begin{remark}
We will write products of basis vectors in the form
\begin{equation}
	\FE_k \FE_l = \delta_{kl} + {c_{kl}}^m \FE_m
	\quad\text{for $k,l = 1\ldots 3$,}
\label{E:ANTICO}
\end{equation}
with ${c_{kk}}^m = 0$ for all indices, using the summation convention. For $k \neq l$, the ${c_{kl}}^m$ are the coefficients in the expansion of the product $\FE_k \FE_l$ with respect to the basis vectors $\FE_m$. Because of \eqref{E:STRUCTURE}, there holds ${c_{kl}}^m + {c_{lk}}^m = 0$ for all indices. We extend the notation consistent with \eqref{E:DUALB}. For example, we want to be able to write
\begin{equation}
\begin{aligned}
	{c_k}^{lm} \FE_m
		& \coloneq \FE_k \FE^l
		= -\FE_k \FE_l
		= -{c_{kl}}^m \FE_m
\\
	c_{klm} \FE^m
		& \coloneq {c_{kl}}^m \FE_m
		= - \sum_m {c_{kl}}^m \FE^m
\end{aligned}
\quad\text{for $k\neq l$,}
\label{E:STRZ}
\end{equation}
which requires ${c_k}^{lm}  = -{c_{kl}}^m$ and $c_{klm} = -{c_{kl}}^m$, by linear independence of $\FE_m$. Thus raising or lowering an index in ${c_{kl}}^m$ results in a change of sign.
\end{remark}

\begin{lemma}
\label{L:COMPLETENESS}
For all $\FP \in \CL{3}$ we have the expansion/basic Fierz identity
\begin{equation}
	\FP = \langle\FE^\mu, \FP\rangle \FE_\mu
		= -\frac{1}{2} \FE^\mu \overline{\FP} \FE_\mu.
\label{E:FIERZ}
\end{equation}
Formula \eqref{E:FIERZ} also holds with $\FE_\mu$ replaced by $\FE^\mu$ and vice versa.
\end{lemma}

Statement \eqref{E:FIERZ} shows that spatial reversal (unlike complex conjugation) can be expressed purely in terms of algebra operations (addition, multiplication). Recall that for $C^*$-algebras the $*$ operation is required as an additional structure.

\begin{proof}
We apply $\langle \FE^\mu, \cdot\rangle$ to $\FP = p^\nu \FE_\nu$, then use linearity, \eqref{E:ORTHO} to obtain $p^\mu = \langle\FE^\mu, \FP\rangle$. Because of definition \eqref{E:SCPR}, $\overline{\FE^\mu} = \FE_\mu$ (see \eqref{E:RAISER}), and $\FE_\mu^2 = 1$, it follows that
\[
	\FP = \frac{1}{2} \left( \FE^\mu\overline{\FP}
			+ \FP\overline{\FE^\mu} \right) \FE_\mu
		= \frac{1}{2} \FE^\mu\overline{\FP}\FE_\mu
			+ \frac{1}{2} \FP \left( \overline{\FE^\mu}\FE_\mu \right)
		= \frac{1}{2} \FE^\mu\overline{\FP}\FE_\mu + 2\FP.
\]
Rearranging terms, we obtain the second statement in \eqref{E:FIERZ}.
\end{proof}

\begin{lemma}
\label{L:COFFS}
For any $\FP \in \CL{3}$ we have that
\[
	(\overline{\FP}\FE^\mu\widehat{\FP})_\nu
		= (\FP\FE_\nu\FP^\dagger)^\mu
	\quad\text{for $\mu,\nu = 0\ldots 3$.}
\]
These are the coefficients of the quantities with respect to the $\FE_\mu$ and $\FE^\nu$ basis vectors, respectively. The same statement holds with $\FE_\mu, \FE^\nu$ replaced by $\FE^\mu, \FE_\nu$.
\end{lemma}

\begin{proof}
We use Lemma~\ref{L:COMPLETENESS}, \eqref{E:SCPR}, and \eqref{E:RAISER} to compute the coefficients
\[
	(\overline{\FP}\FE^\mu\widehat{\FP})_\nu
		= \big( \overline{\FP}\FE^\mu\widehat{\FP} \overline{\FE_\nu} \big)_0
		= \Big( \overline{\overline{\FP}\FE^\mu\widehat{\FP} \overline{\FE_\nu}} \Big)_0
		= \big( \FE_\nu \FP^\dagger\overline{\FE^\mu}\FP \big)_0
		= \big( \FP\FE_\nu\FP^\dagger \overline{\FE^\mu} \big)_0
		= (\FP\FE_\nu\FP^\dagger)^\mu,
\]
where the second equality holds because the scalar part is invariant under spatial reversal, the fourth from rotating the arguments;  see \eqref{E:SPATR}, \eqref{E:ORDER}, and \eqref{E:ROTAT}.
\end{proof}


\subsection{Lorentz transforms}

For $\FL \in \CL{3}$ with $\FL\overline{\FL} = 1$, the map
\begin{equation}
	\FP \mapsto \FL\FP\FL^\dagger
	\quad\text{with $\FP \in \CL{3}$}
\label{E:LORE}
\end{equation}
represents a \emph{Lorentz transform}, with $\FL$ giving the motion and orientation of the object frame with respect to the observer. Note that $\FL\overline{\FL} = 1$ implies $\FL$ has left/right inverse $\FL^{-1} = \overline{\FL}$; see Section~\ref{SS:IE}. The general form is $\FL = \exp(\frac{1}{2}(\BW-i\BTHETA)$, where $\BW, \BTHETA$ are space vectors and $\exp$ is defined in terms of the power series. If $\FL$ is real (thus $\FL=\FL^\dagger$, $\BW \neq 0$, and $\BTHETA = 0$), then $\FL$ represents a boost
\[
	\FB \coloneq \exp(\BW/2) = \cosh(w/2) + \hat{\BW} \sinh(w/2)
\]
in the direction $\vec{w}/w$ with rapidity $w$, where $w \coloneq \|\vec{w}\|$ (recall \eqref{E:EMB}). If $\FL$ is even (thus $\FL = \widehat{\FL}$, $\BTHETA \neq 0$, and $\BW = 0$), then $\FL$ represents a rotation
\[
	\FR \coloneq \exp(-i\BTHETA/2) = \cos(\theta/2) -i \hat{\BTHETA} \sin(\theta/2)
\]
around the axis $\vec{\theta}/\theta$ with angle $\theta$, where $\theta \coloneq \|\vec{\theta}\|$. A Lorentz transform $\FL$ can always be written as the product $\FB\FR$ of a boost $\FB = (\FL\FL^\dagger)^{1/2}$ and a rotation $\FR = \overline{\FB}\FL$. We refer the reader to \cite{Baylis1996} for additional information. The transformation \eqref{E:LORE} preserves the reality of spacetime vectors and their spacetime length because
\[
	\FL\FP\FL^\dagger \, \overline{\FL\FP\FL^\dagger}
		= \FL\FP\FL^\dagger \, \overline{\FL^\dagger} \overline{\FP} \overline{\FL}
		= \FL\FP (\overline{\FL}\FL)^\dagger \overline{\FP} \overline{\FL}
		= (\FP\overline{\FP}) \, \FL\overline{\FL}
		= \FP\overline{\FP},
\]
as follows from $\FL\overline{\FL} = \overline{\FL}\FL = 1$, \eqref{E:ORDER}, and the fact that $\FP\overline{\FP}$ is always a scalar.


\subsection{Left-sided ideals}
\label{SS:LSI}

We define $P \coloneq \frac{1}{2}(1+\FE_3)$, which satisfies
\[
	P + \overline{P} = 1,
	\quad
	P\overline{P} = \overline{P}P = 0,
	\quad\text{and}\quad
	P = P^2 = \widehat{P};
\]
hence $P, \overline{P}$ are projectors. They define the two minimal left-sided ideals
\[
	\mathcal{S}_+ \coloneq \CL{3}P
	\quad\text{and}\quad
	\mathcal{S}_- \coloneq \CL{3}\overline{P}.
\]
The subspace $\mathcal{S}_+$ is spanned by the basis $\{\alpha_0 \coloneq P, \alpha_1 \coloneq \FE_1 P\}$ over $\C$. Indeed since $\FE_2 = i \FE_1\FE_3$ (see \eqref{E:ANTICO}) and $\FE_3 P = P$, for any $\FP = p^\mu \FE_\mu$ with $p^\mu \in \C$ we have
\begin{align*}
	\FP P
		& = (p^0 + p^3) P + (p^1 + ip^2) \FE_1 P,
\end{align*}
which proves the claim. The coefficients of $\FP P$ are thus given as
\[
	(\FP P)^0 = (\FP P)^3 = \frac{p^0 + p^3}{2}
	\quad\text{and}\quad
	(\FP P)^1 = i(\FP P)^2 = \frac{p^1 + ip^2}{2}.
\]
In particular, projection into the left-sided ideal changes the coefficients of all four basis vectors of $\CL{3}$. We observe that $P\FN P = 0$, which implies that
\begin{equation}
	P\FP P = (p^0+p^3) P = 2\FP_0 P
	\quad\text{for all $\FP \in \CL{3}$.}
\label{E:EXPA2}
\end{equation}
The two-sided ideal $P\mathcal{S}_+ = P\CL{3}P$ is therefore one-dimensional as a vector space of $\C$, thus isomorphic to $\C$; it contains only elements of the form $a P$ with $a \in \C$. By symmetry, analogous statements are true that use the projector $\OP$ instead.

For later reference, we record the following observation.

\begin{lemma}
\label{L:CUU}
For any $\FP \in \CL{3}$ we have that
\begin{equation}
	\Big( P \overline{\FP}\FE^\mu\widehat{\FP} P \Big)_0
		= (\widehat{\FP}P\overline{\FP})_\mu
		= (\FP\OP\FP^\dagger)^\mu
	\quad\text{for $\mu=0\ldots 3$.}
\label{E:LIFT}
\end{equation}
\end{lemma}

\begin{proof}
The first equality follows from \eqref{E:ROTAT}, \eqref{E:RAISER}, and $P^2=P$ because
\[
	\Big( P\overline{\FP} \FE^\mu \widehat{\FP}P \Big)_0
		= \Big( \overline{\FE_\mu} (\widehat{\FP}P \overline{\FP}) \Big)_0
		= (\widehat{\FP}P \overline{\FP})_\mu;
\]
see Lemma~\ref{L:COMPLETENESS}. The second equality then follows from the fact that the scalar part is invariant under spatial reversal; see \eqref{E:SPATR} and \eqref{E:ORDER}.
\end{proof}

\begin{remark}
When talking about spin as a model for qubits it is common to use the notation $\FE \equiv \FE_3$ and $\FN \equiv \FE_1$ instead, in which case the third (spatial) basis vector can be written as $i\FN\FE \equiv \FE_2$. The vector $\FE$ represents the direction of spin, whereas $\FN$ is called the NOT-operator. As is evident from the definition $\FE_1$ in \eqref{E:PAULI}, multiplication by $\FN \equiv \FE_1$ flips upper and lower components of a vector in $\C^2$, so that
\[
	\FN \begin{pmatrix} 1 \\ 0 \end{pmatrix} = \begin{pmatrix} 0 \\ 1 \end{pmatrix}
	\quad\text{and vice versa.}
\]
Multiplication by $\FN$ thus flips between the two states representing spin up (logical value $1$) and spin down (logical value $0$), as does the logical operation NOT.
\end{remark}


\subsection{Differential operators}

The $\CL{3}$ analogue of the gradient operator is
\[
	\nabla \coloneq \FE^\mu \partial_\mu
		= \begin{pmatrix} \partial_0-\partial_3 & -\partial_1+i\partial_2 \\
			-\partial_1-i\partial_2 & \partial_0+\partial_3 \end{pmatrix}.
\]
When $\nabla$ is applied to a scalar-valued function $f$ that depends smoothly on the time variable $x^0$ (with $c=1$) and the spatial variables $x^k$, $k=1\ldots 3$, then
\[
    \nabla f = (\partial_0 f) \FE^0 + (\partial_1 f) \FE^1
        + (\partial_2 f) \FE^2 + (\partial_3 f) \FE^3.
\]
Thus $\nabla f$ is the linear combination of partial derivatives $\partial_\mu f$ multiplied by the basis vectors $\FE^\mu$, as for the classical gradient. When applied to a $\CL{3}$-valued function $\phi$, however, then the $\FE^\mu$ interact with the function $\phi$ by algebra multiplication. Here it matters whether $\nabla$ is applied from the left or from the right because the algebra multiplication is not commutative. As we did for elements in $\CL{3}$ we can split the operator $\nabla$ into a time (scalar) and a spatial (vector) part, writing
\begin{equation}
    \nabla = \partial_0 + \BNABLA
    \quad\text{with}\quad
    \BNABLA = \FE^k \partial_k.
\label{E:NADEC}
\end{equation}
Applying the grade automorphism to $\nabla$ acts on the basis vectors, thereby reversing the sign of the vector part of $\nabla$; recall \eqref{E:IDD}. Hence $\widehat{\nabla} = \partial_0 - \boldsymbol{\nabla}$.


\section{Matrix representation}
\label{S:ME}

The most straightforward \emph{representation} of $\CL{3}$ identifies the Clifford algebra with the space $M_2(\C)$ of complex $(2\times 2)$-matrices. The Clifford product is matrix multiplication, and the basis vectors are given by the Pauli matrices
\begin{equation}
    \FE_1 \equiv \begin{pmatrix} 0 & 1 \\ 1 & 0 \end{pmatrix},
    \quad
    \FE_2 \equiv \begin{pmatrix} 0 & -i \\ i & 0 \end{pmatrix},
    \quad
    \FE_3 \equiv \begin{pmatrix} 1 & 0 \\ 0 & -1 \end{pmatrix}.
\label{E:PAULI}
\end{equation}
Indeed one can check that the $\FE_k$ satisfy equations \eqref{E:STRUCTURE} if the right-hand side is understood as multiplied by the unit $(2\times 2)$-matrix $\ONE \equiv \FE_0$. Then
\begin{equation}
    M = M^\mu \FE_\mu
		\equiv \begin{pmatrix} M^0+M^3 & M^1-iM^2 \\ M^1+iM^2 & M^0-M^3 \end{pmatrix}.
\label{E:COEFF}
\end{equation}
The conversion between matrix entries and basis coefficients is given by
\begin{align*}
	M = \begin{pmatrix} A & B \\ C & D \end{pmatrix}
    \quad\Longleftrightarrow\quad
	M^0 &= \frac{A+D}{2}, \quad M^3 = \frac{A-D}{2},
\\
	M^1 &= \frac{C+B}{2}, \quad M^2 = \frac{C-B}{2i}.
\end{align*}
The scalar part of $M$ (the coefficient of $\FE_0 \equiv \ONE$) is thus obtained as half the trace of $M$. It can be convenient to simply work with the representation $M_2(\C)$ instead of $\CL{3}$. Note that all matrices $\FE_\mu$ are Hermitean and the $\FE_k$ have zero trace.

\begin{remark}
\label{R:SFG}
For the Pauli matrices \eqref{E:PAULI} there holds $\FE_1\FE_2 = i\FE_3 = -\FE_2\FE_1$ and cyclic permutations. Therefore the structure relation \eqref{E:ANTICO} takes the form
\begin{equation}
	\FE_k \FE_l = \delta_{kl} - i \varepsilon_{klm} \FE^m
	\quad\text{for $k,l = 1\ldots 3$,}
\label{E:ANTICO2}
\end{equation}
where we have used \eqref{E:STRZ} and the Levi-Civita symbol
\begin{equation}
	\varepsilon_{klm} \coloneq \begin{cases}
		+1 & \text{if $(k,l,m) = (1,2,3)$ or cyclic permutations} \\
		-1 & \text{if $(k,l,m) = (3,2,1)$ or cyclic permutations} \\
		0  & \text{if any two indices are equal}
	\end{cases}.
\label{E:LC}
\end{equation}
In particular, we have $c_{klm} = -i \varepsilon_{klm}$ for $k\neq l$ and $m=1\ldots 3$. We define variants of $\varepsilon_{klm}$ with different combinations of sub/superscripts as we did for $c_{klm}$, keeping in mind that raising/lowering any index requires a sign change.
\end{remark}

In terms of the matrix representation $\CL{3} \equiv M_2(\C)$, we have
\begin{align}
	M = \begin{pmatrix} A & B \\ C & D \end{pmatrix}
    \quad\Longleftrightarrow\quad
	M^\dagger &= \begin{pmatrix} A^* & C^* \\ B^* & D^* \end{pmatrix},
	\quad
	\overline{M} = \begin{pmatrix} D & -B \\ -C & A \end{pmatrix},
\label{E:OVERM}
\\
	\widehat{M} &= \begin{pmatrix} D^* & -C^* \\ -B^* & A^* \end{pmatrix}.
\label{E:DEFMAT}
\end{align}
Thus $M^\dagger$ is indeed the transpose of the complex conjugation of $M$ and $\widehat{M} = \overline{M}^\dagger$. Recall that for a square matrix $M$ the adjugate matrix $\mathrm{adj}(M)$ is defined by
\begin{equation}
	M \mathrm{adj}(M) = \det(M) \ONE.
\label{E:ADJUG}
\end{equation}
While for larger matrices, finding the adjugate matrix involves computing determinants of submatrices, for $(2\times 2)$-matrices it is sufficient to shuffle matrix entries. We have that $\overline{M} = \mathrm{adj}(M)$. Notice that no complex conjugation is involved and the trace is preserved for the adjugate matrix. There holds
\begin{equation}
    M \overline{M} = \overline{M} M = \det(M) \ONE
	\quad\text{with}\quad
	\det(M) = AD-BC.
\label{E:DETM}
\end{equation}
In this case, the multiplicative inverse of $M$ exists because $M\overline{M}$ is a diagonal matrix with nonvanishing entries equal to $\det(M) \neq 0$; see Section~\ref{SS:IE}. In fact $(M\overline{M})^{-1} \overline{M}$ turns out to be just the matrix inverse of $M$ because of \eqref{E:ADJUG}.

It follows from \eqref{E:OVERM} that $M+\overline{M} = (A+D)\ONE$; see \eqref{E:OVERM}. Therefore the scalar product \eqref{E:SCPR} in the Clifford algebra $\CL{3} \equiv M_2(\C)$ takes the form
\[
    \langle M, N \rangle
		\coloneq \frac{1}{2} \TRACE(M\overline{N})
		= \frac{1}{2} \TRACE(N\overline{M})
    \quad\text{for $M,N \in M_2(\C)$,}
\]
where $\TRACE$ is the trace. Recall that $\overline{M\overline{N}} = N \overline{M}$ and $\TRACE(\overline{M}) = \TRACE(M)$; see \eqref{E:ORDER},\eqref{E:OVERM}. The factors of matrix products can be cyclically rotated in the trace; cf.\ \eqref{E:ROTAT}.

Because of \eqref{E:DETM}, the map $M \mapsto LML^\dagger$ for every Hermitean $M \in M_2(\C)$, with $L \in M_2(\C)$, represents a Lorentz transform if and only if
\begin{equation}
	\det(LML^\dagger) = \det(M).
\label{E:INVA}
\end{equation}
Indeed if $M = M^\mu \FE_\mu$ with coefficients $M^\mu \in \R$, then by \eqref{E:DETM} and \eqref{E:COEFF}
\begin{equation}
	\det(M) = (M^0)^2 - \|\vec{M}\|^2.
\label{E:DETT}
\end{equation}
A sufficient condition for \eqref{E:INVA} is that $L \in \mathrm{SL}(2,\C)$, the special linear group of all matrices in $M_2(\C)$ that have determinant equal to one. Since $(LML^\dagger)^\dagger = LML^\dagger$ if $M^\dagger = M$ the Lorentz transform preserves the reality of spacetime vectors.


The projection operators introduced in Section~\ref{SS:LSI} have the form
\[
	P = \begin{pmatrix} 1 & 0 \\ 0 & 0 \end{pmatrix}
	\quad\text{and}\quad
	\OP = \begin{pmatrix} 0 & 0 \\ 0 & 1 \end{pmatrix},
\]
therefore the minimal left-sided ideal $\mathcal{S}_+ = \CL{3}P$ consists of matrices, for which the right column vectors are zero. For $\mathcal{S}_- = \CL{3}\OP$ the left column vectors vanish.


\section{Nonlinear Dirac equation}
\label{S:IDE}

In this section, we will first show how the classical Dirac equation can be reformulated equivalently utilizing the space algebra $\CL{3}$. Then we will discuss a modified, nonlinear equation introduced by Daviau (see \cite{DaviauBertrandSocrounGirardot2022} and references therein), which he calls the improved Dirac equation because it fixes some of the peculiarities of the classical Dirac equation. This modified equation is the cornerstone of a far-reaching attempt by Daviau and coauthors at developing a unified theory of particle physics, as outlined in the book \cite{DaviauBertrandSocrounGirardot2022}. Our exposition is based on Chapter 1 of \cite{DaviauBertrandSocrounGirardot2022}. From now on, we will use the representation $\CL{3} \equiv M_2(\C)$, for notational convenience.

\medskip

We start with the classical Dirac equation
\begin{equation}
    \gamma^\mu D_\mu \psi + im \psi = 0
\label{E:CDIR}
\end{equation}
with $\psi$ a $\C^4$-valued function, called spinor, depending on spacetime variables
\[
    (x^0 = ct, \BX)
    \quad\text{with}\quad
    x^0 \in \R,
    \;
    \vec{x} = \begin{pmatrix} x^1 \\ x^2 \\ x^3 \end{pmatrix} \in \R^3.
\]
Physical constants are the speed of light $c$, the electric charge $e$ and the rest mass $m_0$ of the particle, and the Planck constant $h$, from which we form
\[
    \hbar \coloneq \frac{h}{2\pi},
    \quad
    m \coloneq \frac{m_0 c}{\hbar},
    \quad
    q \coloneq \frac{e}{\hbar c}.
\]
We use the covariant derivative $D_\mu \coloneq \partial_\mu + iq A_\mu$ to couple to an external electromagnetic field, with $A_\mu$ a given electromagnetic potential. We sum over $\mu = 0\ldots 3$ in \eqref{E:CDIR} (Einstein summation convention) and use the Weyl basis, thus
\begin{equation}
    \gamma^0 = \gamma_0 = \begin{pmatrix} 0 & \FE^0 \\ \FE^0 & 0 \end{pmatrix},
    \quad
    \gamma^k = -\gamma_k = \begin{pmatrix} 0 & \FE^k \\ -\FE^k & 0 \end{pmatrix},
    \;
    k = 1\ldots 3,
\label{E:WEYL}
\end{equation}
with $\FE^0 = \FE_0$ and $\FE^k = -\FE_k$, $k=1\ldots 3$, the Pauli matrices \eqref{E:PAULI}. The Weyl basis (also called chiral basis) has the advantage that the spinor $\psi$ can be naturally split into two parts satisfying equations that are only weakly coupled. Indeed let
\[
    \psi \eqcolon \begin{pmatrix} \xi \\ \eta \end{pmatrix}
    \quad\text{with}\quad
    \xi \eqcolon \begin{pmatrix} \xi_1 \\ \xi_2 \end{pmatrix},
    \eta \eqcolon \begin{pmatrix} \eta_1 \\ \eta_2 \end{pmatrix}
\]
taking values in $\C^2$. Then the block structure in \eqref{E:CDIR}/\eqref{E:WEYL} implies that
\begin{equation}
\begin{aligned}
    \FE^\mu (\partial_\mu + iq A_\mu) \eta + im \xi &= 0, \\
    \widehat{\FE^\mu} (\partial_\mu + iq A_\mu) \xi + im \eta &= 0;
\end{aligned}
\label{E:SYST}
\end{equation}
notice the minus sign in the lower left entry in $\gamma^k$ and recall \eqref{E:IDD}. We can combine these two equations into a single one by interpreting \eqref{E:SYST} as the left/right column vectors of a $(2\times 2)$-matrix valued identity. More precisely, let
\[
    R \coloneq \begin{pmatrix} 0 & -\eta_2^* \\ 0 & \eta_1^* \end{pmatrix}
    \quad\text{and}\quad
    L \coloneq \begin{pmatrix} \xi_1 & 0 \\ \xi_2 & 0 \end{pmatrix}
\]
be the left/right parts of the spinor wave. Recall the short-hand notation $\nabla = \FE^\mu \partial_\mu$ and define $A \coloneq A_\mu \FE^\mu$. Then the system \eqref{E:SYST} can be rewritten in the form
\begin{equation}
\begin{aligned}
    (\nabla + iq A) \widehat{R} + im L &= 0, \\
    (\widehat{\nabla} + iq \widehat{A}) L + im \widehat{R} &= 0.
\end{aligned}
\label{E:INTER}
\end{equation}
This is a matrix equation with the right column vectors vanishing; see \eqref{E:DEFMAT}. We apply the grade automorphism, which preserves the order of multiplication factors (see \eqref{E:ORDER}), to the second equation. The sign of $i$ changes under complex conjugation, and we use again \eqref{E:DEFMAT} on $L, R$. Then \eqref{E:INTER} can be rewritten as
\[
\begin{aligned}
    (\nabla + iq A) \widehat{R} + im L &= 0, \\
    (\nabla - iq A) \widehat{L} - im R &= 0.
\end{aligned}
\]
The first matrix equation has vanishing right column vector, for the second matrix equation the left column vector is zero. We can therefore combine the two equations into a single one by adding them, which gives the compact form
\begin{equation}
    \nabla\widehat{\phi} + iq A \widehat{\phi} \FE_3 + im \phi \FE_3 = 0
\label{E:DIRCL}
\end{equation}
for the $\CL{3}$-valued spinor field $\phi \coloneq L + R$. Here we have used multiplication from the right by $\FE_3$ to change the sign of the right column vectors. Notice that \eqref{E:DIRCL} is completely equivalent to the classical Dirac equation \eqref{E:CDIR}/\eqref{E:WEYL}. We can write
\[
	L = \phi P,
	\quad
	R = \phi \OP,
\]
with $P = \frac{1}{2}(1+\FE_3)$ the projector introduced in Section~\ref{SS:LSI}.

\medskip

We multiply \eqref{E:DIRCL} by the spatial reversal $\overline{\phi}$ and obtain
\begin{equation}
    \overline{\phi}\nabla\widehat{\phi}
		+ iq (\overline{\phi} A \widehat{\phi}) \FE_3
		+ im (\overline{\phi} \phi) \FE_3 = 0.
\label{E:INVDIR}
\end{equation}
As discussed in Section~\ref{S:SA}, this equation is equivalent to \eqref{E:DIRCL} as long as $\phi$ is invertible. Consider now a map $x \mapsto MxM^\dagger$ for some invertible complex $(2\times 2)$-matrix $M$. Suppose further that there exists another spinor field $\psi$ such that
\begin{equation}
	\phi(x) = M^{-1} \psi(y)
	\quad\text{with}\quad
	y = MxM^\dagger,
\label{E:TRAFO}
\end{equation}
for all spacetime vectors $x$. Using \eqref{E:ORDER} and the chain rule, we get that
\begin{equation}
	\overline{\phi}(x) \big( \nabla \widehat{\phi}(x) \big)
		= \overline{\phi}(x) \FE^\mu \frac{\partial \widehat{\phi}(x)}{\partial x^\mu}
		= \overline{\psi}(y) \Big( \overline{M^{-1}} \FE^\mu \widehat{M^{-1}} \Big)
			\frac{\partial\widehat{\psi}(y)}{\partial y^\kappa}
				\frac{\partial y^\kappa}{\partial x^\mu}.
\label{E:MEINS}
\end{equation}
Writing $x = x^\mu \FE_\mu$, we obtain $Mx M^\dagger = x^\mu D_\mu^\kappa \FE_\kappa$ with $D^\kappa_\mu \coloneq \langle \FE^\kappa, M\FE_\mu M^\dagger \rangle$. Thus
\begin{equation}
	y^\kappa = x^\mu D^\kappa_\mu
	\quad\text{and}\quad
	\frac{\partial y^\kappa}{\partial x^\mu} = D^\kappa_\mu
	\quad\text{for $\kappa, \mu = 0\ldots 3$.}
\label{E:MZWEI}
\end{equation}
Because of Lemma~\ref{L:COFFS}, there holds $D^\kappa_\mu = \langle \FE_\mu, \overline{M} \FE^\kappa \widehat{M} \rangle$, hence
\begin{equation}
	D^\kappa_\mu \FE^\mu = \overline{M} \FE^\kappa \widehat{M}
	\quad\text{for $\kappa = 0\ldots 3$.}
\label{E:MDREI}
\end{equation}
Then we use \eqref{E:MZWEI}/\eqref{E:MDREI} and $\overline{M^{-1}} = \overline{M}^{-1}$ in \eqref{E:MEINS} to conclude that
\[
	\overline{\phi}(x) \big( \nabla\widehat{\phi}(x) \big)
		= \overline{\psi}(y) \FE^\kappa \frac{\partial \widehat{\psi}(y)}{\partial y^\kappa}
		= \overline{\psi}(y) \big( \nabla\widehat{\psi}(y) \big).
\]
The derivative part of the Dirac equation \eqref{E:INVDIR} is therefore form invariant under the transformations \eqref{E:TRAFO}, which include the Lorentz transforms with $M \in \mathrm{SL}(2,\C)$, scalings with $M = e^{i\beta/2}\ONE$ for $\beta \in \R$, and $M = r \ONE$ with $r > 0$. The potential term in \eqref{E:INVDIR} transforms in an analogous way to preserve gauge invariance.

Using \eqref{E:TRAFO} in the mass term in \eqref{E:INVDIR}, we obtain
\[
	\overline{\phi}(x) \phi(x)
		= \overline{\psi}(y) \big( \overline{M} M) \psi(y)
		= \det(M) \, \overline{\psi}(y) \psi(y),
\]
so this term is form invariant if $M$ is a Lorentz transform, i.e., if $M \in \mathrm{SL}(2,\C)$. It is, however, not form invariant if $M = e^{i\beta/2}\ONE$ with $\beta \neq 0$, in which case $\det(M) = e^{i\beta}$. For this choice of $M$, spacetime does not change at all:
\[
	y = MxM^\dagger = x
	\quad\text{for all $x$,}
\]
only the spinor is multiplied by some global exponential factor. It would be desirable to have this invariance (called $\mathrm{U}(1)$ invariance) as well. It seems that this requires nonlinearity. There are various nonlinear extensions of the Dirac equations, like the Soler model \cites{Ivanenko1938, Soler1970}, for which the mass term becomes \emph{cubic} in the spinor.

\medskip

Following \cite{DaviauBertrandSocrounGirardot2022}, we will consider here the nonlinear Dirac equation
\begin{equation}
	\overline{\phi} \nabla\widehat{\phi}
		+ iq (\overline{\phi}A \widehat{\phi})\FE_3 + im N \FE_3 = 0
	\quad\text{with}\quad
	N \coloneq |\det(\phi)|.
\label{E:IMPINV}
\end{equation}
Recall that $\overline{\phi} \phi = \det(\phi)\ONE$, which typically is complex-valued. The novelty in \eqref{E:IMPINV} is to use the absolute value of the determinant. It is straightforward to check that $N$ is indeed $\mathrm{U}(1)$ invariant. Equation \eqref{E:IMPINV} is invariant under scaling $M = r\ONE$ with $r>0$ if the particle mass transforms according to $m \mapsto m |\det(M)|$. We refer the reader to \cite{DaviauBertrandSocrounGirardot2022} for further discussion and motivation of \eqref{E:IMPINV}.

The nonlinearity in \eqref{E:IMPINV} is benign in the sense that it has the same homogeneity as the original term $\overline{\phi} \phi$ (unlike the cubic Soler model). On the other hand, it is not differentiable. It can be rewritten in the following instructive form. We have
\[
	N \ONE
		= |\det(\phi)| \ONE
		= \Big( (\overline{\phi}\phi) (\overline{\phi}\phi)^\dagger \Big)^{1/2}
		= N^{-1} \Big( (\overline{\phi}\phi) (\overline{\phi}\phi)^\dagger \Big)
		= N^{-1} \Big( \overline{\phi} \big(\phi\phi^\dagger\big) \widehat{\phi} \Big).
\]
Notice that $\overline{\phi}\phi = \det(\phi) \ONE$ and the square root of a positive semidefinite diagonal matrix is well-defined. We now define the Dirac current $J \coloneq \phi\phi^\dagger$ and obtain
\begin{equation}
	N \ONE = \overline{\phi}V\widehat{\phi}
	\quad\text{with velocity}\quad
	V \coloneq N^{-1} J.
\label{E:PILOT}
\end{equation}
Decomposing $J = \FE^\mu J_\mu$, we have $J_0 \GS 0$ because $J$ is Hermitean positive semidefinite. As a consequence of \eqref{E:PILOT} and $P\GS 0$, we have that
\begin{equation}
	V_0 \GS 0
	\quad\text{where}\quad
	V = \FE^\mu V_\mu
\label{E:COMPA}
\end{equation}
(no negative energy). Defining $\det(\phi) \eqcolon N e^{i\beta}$, we can also write
\[
	N e^{-i\beta} \ONE = \big( \det(\phi) \big)^* \ONE
		= \big( \det(\phi) \ONE \big)^\dagger
		= (\overline{\phi}\phi)^\dagger
		= \phi^\dagger \widehat{\phi}.
\]
Multiplying from the left by $\phi$ and dividing by $P$ (assumed nonzero) we obtain
\[
	e^{-i\beta} \phi = N^{-1} \phi (\phi^\dagger \widehat{\phi})
		= N^{-1} (\phi \phi^\dagger) \widehat{\phi}
		= V \widehat{\phi}.
\]
The factor $V$ can therefore be used to perform the grade automorphism $\phi \mapsto \widehat{\phi}$, up to a complex exponential factor depending on $\phi$. Note that the grade automorphism is not a linear map because it involves a complex conjugation.

\medskip

The nonlinear Dirac equation \eqref{E:IMPINV} now takes the form
\begin{equation}
	\overline{\phi} \nabla\widehat{\phi}
		+ i\overline{\phi} \Big( qA + mV \Big) \widehat{\phi}\FE_3 = 0.
\label{E:ALMO}
\end{equation}
Assuming that $\phi$ is invertible, we can multiply \eqref{E:ALMO} from the left by $\phi$ and cancel the numerical factor $\det(\phi)$ to obtain the remarkably simple equation
\begin{equation}
\tcboxmath{ 
	\nabla\widehat{\phi} + i \Big( qA + mV \Big) \widehat{\phi} \FE_3 = 0.
} 
\label{E:DIRACEQ}
\end{equation}

The Dirac current $J = \phi\phi^\dagger$ is Hermitean positive semidefinite. We have
\begin{equation}
	N^2
	= |\det(\phi)|^2
	= \frac{1}{2} \TRACE\Big( (\overline{\phi}\phi) (\overline{\phi}\phi)^\dagger \Big)
	= \frac{1}{2} \TRACE\Big( (\phi\phi^\dagger) \overline{(\phi\phi^\dagger)} \Big)
	= \det(J),
\label{E:SAME}
\end{equation}
rotating trace arguments cyclically. The matrices in both traces are diagonal because of \eqref{E:ADJUG}. As a consequence, the velocity $V$ satisfies
\begin{equation}
	\det(V) = \det(N^{-1}J) = N^{-2} \det(J) = 1,
\label{E:NORMA}
\end{equation}
hence $V$ is timelike. Borrowing an idea of De Broglie, we think of $V$ as a pilot wave guiding the particle. Unlike the electromagnetic potential $A$, the field $V$ is not given externally but is defined implicitly in terms of the solution $\phi$ of \eqref{E:DIRACEQ} itself.


\subsection{Plane waves}

We are looking for solutions of \eqref{E:DIRACEQ} of the form
\begin{equation}
	\phi(x)
	= M e^{-i\varphi(x)\FE_3}
	= M \begin{pmatrix} e^{-i\varphi(x)} & 0 \\
		0 & e^{i\varphi(x)} \end{pmatrix},
	\quad
	x = (x^\mu) \in \R^4
\label{E:PW}
\end{equation}
for some $M \in \CL{3}$ and a real-valued function $\varphi$ to be determined. Then
\begin{gather*}
	N(x) = \big| \det\big( \phi(x) \big)\big| = |\det(M)|,
\\
	J(x) = \phi(x)\phi(x)^\dagger = MM^\dagger
\end{gather*}
are constant. The field $V \coloneq N^{-1} J$ appears in the nonlinear Dirac equation \eqref{E:DIRACEQ}, which therefore has constant coefficients (the electromagnetic potential is assumed to vanish for our construction). Using the ansatz \eqref{E:PW}, we observe that
\begin{align*}
	\nabla \widehat{\phi}(x)
		& = \FE^\mu \partial_\mu \Bigg(
			\widehat{M}
			\begin{pmatrix} e^{-i\varphi(x)} & 0 \\
				0 & e^{i\varphi(x)} \end{pmatrix} \Bigg)
\\
		& = -i \Big( \FE^\mu \partial_\mu \varphi(x) \Big) \Bigg(
			\widehat{M}
			\begin{pmatrix} e^{-i\varphi(x)} & 0 \\
				0 & -e^{i\varphi(x)} \end{pmatrix} \Bigg)
		= -i \Big( \FE^\mu \partial_\mu \varphi(x) \Big)
			\widehat{\phi}(x)\FE_3;
\end{align*}
recall \eqref{E:DEFMAT}/\eqref{E:ORDER}. Then $\phi$ is a solution of \eqref{E:DIRACEQ} provided that
\[
	\FE^\mu \partial_\mu \varphi(x) = mV
	\quad\text{hence}\quad
	\varphi(x) = m V_\mu x^\mu + \varphi_0,
\]
with $V = \FE^\mu V_\mu$ and $\varphi_0 \in \R$ some number. By \eqref{E:DETT}/\eqref{E:NORMA}, we have that
\[
	1 = \det(V) = V_0^2 - \|\vec{V}\|^2,
	\quad\text{with}\quad
	\vec{V} = (V_k) \in \R^3
\]
and $\|\cdot\|$ the Euclidean norm. It follows that $V_0 = \sqrt{1 + \|\vec{V}\|^2}$; see \eqref{E:COMPA}.

We emphasize the fact that the velocity field $V$ and the prefactor $M$ in \eqref{E:PW} are not independent of each other; this is different for the classical Dirac equation. In order to show how the quantities $N, J$ defined as
\begin{equation}
	N = |\det(M)|
	\quad\text{and}\quad
	J = MM^\dagger
\label{E:PV2}
\end{equation}
can be computed from (the complex matrix) $M$, consider
\[
	M = \begin{pmatrix} A & B \\ C & D \end{pmatrix}
	\quad\text{hence}\quad
	J = \begin{pmatrix} |A|^2+|B|^2 & AC^*+BD^* \\
		CA^*+DB^* & |C|^2+|D|^2 \end{pmatrix}
\]
and $N = |AD-BC|$. Expanding $J = \FE_\mu J^\mu$, we find that
\[
\begin{alignedat}{3}
	J^0 &= \frac{|A|^2 + |B|^2 + |C|^2 + |D|^2}{2},
	& \quad
	J^1 &= \frac{AC^* + BD^* + CA^* + DB^*}{2},
\\
	J^2 &= \frac{AC^* + BD^* - CA^* - DB^*}{2i},
	& \quad
	J^3 &= \frac{|A|^2 + |B|^2 - |C|^2 - |D|^2}{2}.
\end{alignedat}
\]
Then $J_0 = J^0$ and $J_k = -J^k$ because of \eqref{E:RAISER}. Note that indeed $J_0 \GS 0$.

For the converse direction, we first observe that we can always write $M = SU$ with $U$ unitary and $S$ Hermitean positive semidefinite. This is the polar decomposition of matrices. The quantities $N, J$ in \eqref{E:PV2} do not change when $M$ is replaced by $M U^\dagger$ for any $U$ unitary. Given $N, J$ we thus look for $M$ Hermitean positive semidefinite such that \eqref{E:PV2} holds. The matrix $M$ can be written as
\[
	M = \begin{pmatrix} A & Z \\ Z^* & D \end{pmatrix}
	\quad\text{hence}\quad
	J = \begin{pmatrix} A^2+|Z|^2 & (A+D)Z \\
		(A+D)Z^* & D^2+|Z|^2 \end{pmatrix},
\]
with $A,D \GS 0$, $Z\in\C$, and $AD \GS |Z|^2$. Using again $J = \FE_\mu J^\mu$, we obtain
\begin{equation}
\begin{alignedat}{3}
	J^0 &= \frac{A^2 + D^2 + 2|Z|^2}{2},
	& \quad
	J^1 &= \frac{(A+D) (Z+Z^*)}{2},
\\
	J^2 &= \frac{(A+D) (Z-Z^*)}{2i},
	& \quad
	J^3 &= \frac{A^2 - D^2}{2}.
\end{alignedat}
\label{E:SYST2}
\end{equation}
In order to solve this system for $A,B,Z$, we introduce new variables
\begin{equation}
	\left.\begin{alignedat}{3}
		M^0 &= \frac{A+D}{2},
		& \quad
		M^1 &= \frac{Z+Z^*}{2} \\
		M^2 &= \frac{Z-Z^*}{2i},
		& \quad
		M^3 &= \frac{A-D}{2}
	\end{alignedat}\right\}
	\quad\Longleftrightarrow\quad
	\left\{\begin{aligned}
		A &= M^0+M^3 \\
		D &= M^0-M^3 \\
		Z &= M^1+iM^2
	\end{aligned}\right.
\label{E:SECOND}
\end{equation}
which are the coefficients of $M$ in the basis $\FE_\mu$. Then \eqref{E:SYST2} takes the form
\begin{equation}
	J^0 = (M^0)^2 + (M^1)^2 + (M^2)^2 + (M^3)^2
	\quad\text{and}\quad
	J^k = 2M^0 M^k.
\label{E:FIRST}
\end{equation}
We substitute $M^k = J^k/(2M^0)$, then multiply by $(M^0)^2$ to obtain
\[
	\bigg( (M^0)^2 - \frac{J^0}{2} \bigg)^2
		= \bigg( \frac{J^0}{2} \bigg)^2
		- \bigg( \frac{J^1}{2} \bigg)^2
		- \bigg( \frac{J^2}{2} \bigg)^2
		- \bigg( \frac{J^3}{2} \bigg)^2
		= \frac{N^2}{4};
\]
see \eqref{E:DETT}/\eqref{E:SAME}. Since $0\LS N \LS J^0$, this can be solved for $M^0$ to give
\begin{equation}
	M^0 = \sqrt{\frac{J^0 \pm N}{2}}.
\label{E:REST2}
\end{equation}
Now $M^k$ can be computed from \eqref{E:FIRST}, after which $A,D,Z$ follow from \eqref{E:SECOND}. Notice that $M^0$ in \eqref{E:REST2} must indeed be nonnegative because $M^0 = \frac{1}{2}\TRACE(M)$ and $M$ is assumed Hermitean positive semidefinite. In \eqref{E:REST2} we prefer the $+$ sign for better numerical stability. We have $J^0 = N$ if $J^k = 0$ for $k=1\ldots 3$. It is straightforward to double check that \eqref{E:PV2} holds with this choice of parameters. Moreover $\det(M) \GS 0$ because $M$ is Hermitean positive semidefinite.


\subsection{Conservation laws}

Arguably, the quadratic form \eqref{E:ALMO} of the Dirac equation is the more fundamental one because only bilinear quantities formed by the spinor are physically accessible. Recall from Section~\ref{S:IDE} that \eqref{E:ALMO} is form invariant under a large class of transformations that includes the Lorentz transforms. While we found it convenient to consider $\CL{3}$ as a vector space over the complex numbers $\C$, it is important to remember that elements in the algebra split naturally into real and imaginary parts; recall the basis \eqref{E:BASIS} and the discussion in Section~\ref{S:INVO}. Each part has a separate expansion with respect to the basis vectors $\FE_\mu$.

The Dirac equation \eqref{E:ALMO} therefore amounts to \emph{eight} independent scalar equation that must be satisfied simultaneously: The coefficients of the real and imaginary parts in the expansion with respect to the basis vectors $\FE_\mu$ must be zero, by linear independence of basis vectors. To simplify the notation, we introduce
\[
	\Gamma \coloneq qA+mV
	\quad\text{with}\quad
	\Gamma \eqcolon \Gamma_\nu \FE^\nu.
\]
We also multiply \eqref{E:ALMO} by $-i$ to bring the imaginary unit to the differential operator, which is the correct form of the momentum operator of quantum mechanics. The following discussion is based on Sections~1.4/5 in \cite{DaviauBertrandSocrounGirardot2022}.

We then multiply by $\FE^\kappa$ and take the imaginary part, so that
\begin{equation}
	0 = \IM\Big( \FE^\kappa \OPHI \FE^\mu (-i\partial_\mu\WPHI) \Big)_0
		+ \Gamma_\mu \IM\Big( \FE^\kappa \OPHI \FE^\mu\WPHI\FE_3 \Big)_0.
\label{E:TWOIM}
\end{equation}
For the first term, we use \eqref{E:REIM}, \eqref{E:ROTAT} to obtain
\begin{align}
	\IM\Big( \FE^\kappa \OPHI \FE^\mu (-i\partial_\mu\WPHI) \Big)_0
		& = \frac{1}{2i} \bigg( \Big( \FE^\kappa \OPHI \FE^\mu (-i\partial_\mu\WPHI) \Big)
			- \Big( \FE^\kappa \OPHI \FE^\mu (-i\partial_\mu\WPHI) \Big)^\dagger \bigg)_0
\label{E:DIVC}
\\
		& = -\frac{1}{2} \bigg( \FE^\mu \Big( (\partial_\mu\WPHI)\FE^\kappa\OPHI
			+ \WPHI\FE^\kappa(\partial_\mu\OPHI) \Big) \bigg)_0
		= -\frac{1}{2} \partial_\mu \Big( \FE^\mu (\WPHI\FE^\kappa\OPHI) \Big)_0.
\nonumber
\end{align}
Because of \eqref{E:ORDER}, \eqref{E:RAISER} we can write $\WPHI\FE^\kappa\OPHI = \overline{\phi\FE_\kappa\phi^\dagger}$. Defining the currents
\begin{equation}
	D_\kappa \coloneq \phi\FE_\kappa\phi^\dagger
	\quad\text{for $\kappa=0\ldots 3$,}
\label{E:CURRY}
\end{equation}
we compute their coefficients with respect to the basis vectors $\FE_\mu$ as
\begin{equation}
	{D_\kappa}^\mu
		= \langle \FE^\mu, D_\kappa \rangle
		= \Big( \FE^\mu \overline{\phi\FE_\kappa\phi^\dagger} \Big)_0
		= \Big( \FE^\mu (\WPHI\FE^\kappa\OPHI) \Big)_0,
\label{E:CURRO}
\end{equation}
which we insert in \eqref{E:DIVC}. For the second term in \eqref{E:TWOIM}, we write
\begin{align*}
	\IM\Big( \FE^\kappa \OPHI \FE^\mu \WPHI \FE_3 \Big)_0
		& = \frac{1}{2i} \bigg( \Big( \FE^\kappa \OPHI \FE^\mu \WPHI \FE_3 \Big)
			- \Big( \FE^\kappa \OPHI \FE^\mu \WPHI \FE_3 \Big)^\dagger \bigg)_0
\\
		& = \frac{1}{2i} \bigg( \FE^\mu \Big( \WPHI(\FE_3\FE^\kappa -\FE^\kappa\FE_3) \OPHI \Big) \bigg)_0.
\end{align*}
If $\kappa = 0$ or $\kappa = 3$, then this term vanishes. Otherwise, we have
\begin{gather*}
	\FE_3\FE^1 = -\FE_3\FE_1 = -i\FE_2 = i\FE^2 = -\FE^1\FE_3,
\\
	\FE_3\FE^2 = -i\FE^1 = -\FE^2\FE_3.
\end{gather*}
Using again \eqref{E:CURRO}, we arrive at the conservation laws
\begin{equation}
\begin{alignedat}{4}
	\partial_\mu {D_0}^\mu & = 0 \qquad & \partial_\mu {D_1}^\mu & = +2\Gamma_\mu {D_2}^\mu
\\
	\partial_\mu {D_3}^\mu & = 0 \qquad & \partial_\mu {D_2}^\mu & = -2\Gamma_\mu {D_1}^\mu
\end{alignedat}
\label{E:FIRP}
\end{equation}
which is the first set of four equations that follow from the Dirac equation \eqref{E:ALMO}.
The conservation law $D_3$ has no correspondence for the classical Dirac equation. It is a consequence of the additional $\text{U}(1)$ invariance of the nonlinear Dirac equation. Notice that any linear combination of the currents $D_0$ and $D_3$ is still a conserved quantity. As a special case, we will write $J \coloneq D_0$ for the Dirac current.

We now multiply \eqref{E:ALMO} by $-i\FE^\kappa$ and take the real part to obtain
\begin{equation}
	0 = \RE\Big( \FE^\kappa \OPHI \FE^\mu (-i\partial_\mu\WPHI) \Big)_0
		+ \Gamma_\mu \RE\Big( \FE^\kappa \OPHI \FE^\mu\WPHI\FE_3 \Big)_0.
\label{E:CURSO}
\end{equation}
Then we compute
\begin{align*}
	\RE\Big( \FE^\kappa \OPHI \FE^\mu (-i\partial_\mu\WPHI) \Big)_0
		& = \frac{1}{2} \bigg( \Big( \FE^\kappa \OPHI \FE^\mu (-i\partial_\mu\WPHI) \Big)
			+ \Big( \FE^\kappa \OPHI \FE^\mu (-i\partial_\mu\WPHI) \Big)^\dagger \bigg)_0
\\
		& = -\frac{i}{2} \bigg( \FE^\mu \Big( (\partial_\mu\WPHI) \FE^\kappa \OPHI
			- \WPHI \FE^\kappa (\partial_\mu\OPHI) \Big) \bigg)_0.
\end{align*}
For the second term in \eqref{E:CURSO}, we use \eqref{E:DUALB}, \eqref{E:CURRO} again to obtain
\begin{align*}
	\RE\Big( \FE^\kappa \OPHI \FE^\mu\WPHI\FE_3 \Big)_0
		& = \frac{1}{2} \bigg( \Big( \FE^\kappa \OPHI \FE^\mu\WPHI\FE_3 \Big)
			+ \Big( \FE^\kappa \OPHI \FE^\mu\WPHI\FE_3 \Big)^\dagger \bigg)_0
\\
		& = \frac{1}{2} \bigg( \FE^\mu \Big( \WPHI (\FE_3 \FE^\kappa+\FE^\kappa \FE_3) \OPHI \Big) \bigg)_0.
\end{align*}
This term vanishes if $\kappa = 1$ or $\kappa = 2$. Otherwise, we have
\[
	\RE\Big( \OPHI \FE^\mu\WPHI\FE_3 \Big)_0 = -{D_3}^\mu
	\quad\text{and}\quad
	\RE\Big( \FE^3 \OPHI \FE^\mu\WPHI\FE_3 \Big)_0 = -{D_0}^\mu
\]
because of \eqref{E:DUALB}, \eqref{E:CURRO}. We can now insert these terms in \eqref{E:CURSO}, thereby obtaining the second set of four equations that follow from \eqref{E:ALMO}. Note that these equations together with \eqref{E:FIRP} are indeed equivalent to \eqref{E:ALMO} as they describe the coefficients of the real and imaginary parts of the Dirac equation in the $\FE_\mu$ basis.

Since $P + \OP = 1$ the Dirac current naturally splits
\[
	J = \phi\phi^\dagger
	= \phi P\phi^\dagger + \phi\OP\phi^\dagger
	\eqcolon \FJ_L + \FJ_R
\]
into the left and right chiral current, respectively, where
\[
	\FJ_L = \frac{1}{2} (D_0+D_3)
	\quad\text{and}\quad
	\FJ_R = \frac{1}{2} (D_0-D_3).
\]
In the same way, the energy-momentum tensor naturally splits into left and right contributions. Conversely, quantities for the Dirac particle can be reassembled from its left and right constituents. We refer the reader to  \cite{DaviauBertrandSocrounGirardot2022} for further discussion.

Let us also record the orthogonality relation
\[
	\langle D_\mu, D_\nu \rangle = N^2 \eta_{\mu\nu}
	\quad\text{for $\mu,\nu = 0\ldots 3$,}
\]
which follows from the fact that $\OPHI\phi = \det(\phi)$ (see \eqref{E:ADJUG}) and
\[
	\langle D_\mu, D_\nu \rangle
		= \Big( \phi\FE_\mu\phi^\dagger \overline{\phi\FE_\nu\phi^\dagger} \Big)_0
		= \Big( \FE_\mu (\OPHI\phi)^\dagger \overline{\FE_\nu} (\OPHI\phi)\Big)_0
		= |\det(\phi)|^2 \langle \FE_\mu,\FE_\nu \rangle.
\]
Here we have used \eqref{E:CURRY} in \eqref{E:SCPR}, \eqref{E:ORDER}, \eqref{E:MINKO}, and \eqref{E:IMPINV}.


\subsection{Global existence}

We prove global existence of solutions to a regularized version of the Dirac equation \eqref{E:DIRACEQ}. The regularization is necessary because the nonlinearity, which is homogeneous of degree one, is non-smooth in nodal points where $\det(\phi)$ vanishes. Our regularization replaces the nonlinearity by a \emph{Lipschitz-continuous} function of the spinor field $\phi$. Then we can apply the standard theory of semilinear evolution equations. Existence for the nonlinear Dirac equation without regularization will be explored elsewhere.

Let us start by introducing the relevant function spaces. Since we use the matrix representation of $\CL{3}$ introduced in Section~\ref{S:ME}, where elements of the Clifford algebra are identified with matrices in $M_2(\C)$, the Frobenius inner product induces a norm $\|\cdot\|$ on $\CL{3}$. We denote by $L^2(\R^3; M_2(\C))$ the Hilbert space of square integrable functions, and by $H^1(\R^3; M_2(\C))$ the associated Sobolev space of first order. For simplicity of notation, we will just write $L^2$ and $H^1$. We denote by $L^\infty$ the space of essentially bounded functions on $\R^3$, equipped with the canonical norm.

In the following, if $a$ and $b$ are two nonnegative quantities, we will write $a \lesssim b$ to indicate that $a \leq bc$ for some universal constant $c>0$.

\medskip

In order to define the regularized Dirac equation, we first observe that
\[
	V \widehat{\phi}
		= \big( N^{-1} \phi \phi^\dagger\big) \WPHI
		= \big( N^{-1} \det(\phi)^* \big) \phi
	\quad\text{with}\quad
	N = |\det(\phi)|
\]
because $\phi^\dagger \widehat{\phi} = (\overline{\phi} \phi)^\dagger = \det(\phi)^*$; recall \eqref{E:ADJUG}. For $\lambda > 0$ we define
\begin{equation}
	V_{\lambda}(\phi) \coloneq \frac{\det(\phi)^{\ast}}{N+\lambda\|\phi\|^2}
\label{E:REGV}
\end{equation}
for all $\phi \neq 0$. Then $|V_\lambda(\phi)| \LS 1$ for all $\phi \neq 0$, even in the limit $\| \phi \| \to 0$. Recalling \eqref{E:DIRACEQ}, we now consider the regularized Dirac equation
\begin{equation}
	\nabla \WPHI + i\Big( qA \WPHI + m V_\lambda(\phi) \phi \Big) \FE_3 = 0,
\label{E:DIRACREG}
\end{equation}
with ${A}$ an external electromagnetic potential specified below. Introducing
\begin{equation}
	F_{\lambda}(\phi) \coloneq \big( V_{\lambda}(\phi)-1 \big) \phi,
\label{E:SOUR}
\end{equation}
we can rewrite \eqref{E:DIRACREG} in the form
\[
 	\nabla \widehat{\phi} + i \Big( q A \widehat{\phi} + m \phi \Big) \FE_3
		+ i m F_{\lambda}(\phi) \FE_3 = 0.
\]
The first two terms on the left-hand side are precisely the ones from the classical Dirac equation \eqref{E:DIRCL}. Here the $F_\lambda(\phi)$-term acts as a source.

\begin{lemma}
\label{lem:estimate}
Given $\lambda>0$, let $\phi \mapsto V_\lambda(\phi)$ be defined by \eqref{E:REGV}. Then
\[
	|\partial_\phi V_{\lambda}(\phi)|
		\lesssim \frac{1+\lambda}{N + \lambda \|\phi\|^2} \|\phi\|
	\quad\text{for all $\phi$.}
\]
In particular, its derivative is bounded and so $V_\lambda$ is Lipschitz continuous.
\end{lemma}

\begin{proof}
By the product/quotient rule, it follows from \eqref{E:REGV} that
\[
	|\partial_\phi V_{\lambda}(\phi)|
		\LS \frac{|\partial_\phi \det(\phi)|}{N + \lambda\|\phi\|^2}
		+ \frac{N}{N+\lambda\|\phi\|^2}
			\frac{|\partial_\phi N|+\lambda|\partial_\phi \|\phi\|^2|}{N + \lambda\|\phi\|^2}.
\]
Using Jacobi's formula for the derivative of a determinant, we have
\[
	|\partial_\phi \det({\phi})|
		\LS \| \mathrm{adj}(\phi) \|
		= \|\phi\|,
\]
with $\mathrm{adj}(\phi)$ the adjugate matrix; recall \eqref{E:OVERM}. Similarly
\[
	|\partial_\phi N|
		= \big| \partial_\phi|\det(\phi)| \big|
		\LS |\partial_\phi\det(\phi)|
		\LS \|\phi\|.
\]
Using the Cauchy-Schwarz inequality, we find that
\[
	\big| \partial_\phi \|\phi\|^2 \big|
		\LS 2 \|\phi\|.
\]
Combining these estimates, we obtain the stated result.
\end{proof}

\begin{lemma}
\label{lem:nonlinearity_properties}
Given $\lambda> 0$, let $\phi \mapsto F_\lambda(\phi)$ be defined by \eqref{E:SOUR}. Then
\begin{enumerate}
\item $F_\lambda \colon L^2 \longrightarrow L^2$ with
\[
	\|F_\lambda(\phi)\|_{L^2} \LS 2 \|\phi\|_{L^2};
\]
\item $F_\lambda \colon H^1 \longrightarrow H^1$ with
\[
	\|\nabla F_\lambda(\phi)\|_{L^2} \lesssim (\lambda^{-1}+1) \|\nabla \phi\|_{L^2};
\]
\item $F_\lambda$ is Lipschitz continuous on $L^2$ with
\[
	\|F_\lambda(\phi)-F_\lambda(\phi')\|_{L^2} \lesssim (1+\lambda^{-1} )\|\phi-\phi'\|_{L^2}.
\]
\end{enumerate}
\end{lemma}

\begin{proof}
Item~(1) follows immediately from the definition of $F_\lambda(\phi)$ and the fact that $|V_{\lambda}(\phi)|\LS 1$ for all $\phi\neq 0$. For item~(2), we can estimate a.e.
\begin{equation}
	\|\partial_k F_\lambda(\phi)\|
		\LS \big\| \big( \partial_k V_\lambda(\phi) \big) \phi \big\|
			+ \big\| \big( V_{\lambda}(\phi)-1 \big) \partial_k{\phi} \big\|
\label{E:PROD}
\end{equation}
for $k=1\ldots 3$, using the product/chain rule. By Lemma~\ref{lem:estimate}, we have a.e.
\[
	\big\| \big( \partial_k V_\lambda(\phi) \big) \phi \big\|
		\LS |\partial_k V_\lambda(\phi)| \|\phi\|
		\lesssim \frac{1+\lambda}{N + \lambda\|\phi\|^2} \|\phi\|^2 \|\partial_k\phi\|
		\LS (\lambda^{-1}+1) \|\partial_k\phi\|.
\]
Taking the square in \eqref{E:PROD} and integrating over $\R^3$, we obtain
\[
	\|\partial_k F_\lambda(\phi)\|_{L^2}
		\lesssim (\lambda^{-1}+1) \|\partial_k\phi\|_{L^2},
\]
where we have used again that $|V_{\lambda}(\phi)|\LS 1$. For item~(3), we argue as we did for (2) to estimate the derivative of $\phi \mapsto F_\lambda(\phi)$ in terms of $\lambda^{-1}+1$, then use the mean value theorem to establish the Lipschitz continuity of the map. Inserting the spinor field $\phi$ and integrating over $\R^3$, we obain the result.
\end{proof}

We decompose the electromagnetic potential $A$ and the operator $\nabla$ into their time (scalar) and spatial (vector) parts: $A \eqcolon A_0 + \BA$ and $\nabla \eqcolon \partial_0 + \BNABLA$, then rewrite the regularized Dirac equation (multiplied by $-i$) in the form
\[
	i\partial_0 \WPHI = H \WPHI + m F_\lambda(\phi),
\]
with Hamiltonian operator $H$ defined by
\[
	H\WPHI := \Big( -i\BNABLA \WPHI + q\BA \WPHI \FE_3 \Big) + \Big( qA_0 \WPHI + m\phi\FE_3 \Big).
\]
Note that the map $\phi \mapsto \WPHI$ is an isomorphism for both $L^2$ and $H^1$; recall \eqref{E:DEFMAT}.

Our goal is to prove global existence of the regularized Dirac equation: Fix $T>0$ and initial data $\phi_0 \in H^1$. Then find $\phi \in C([0,T];H^1) \cap C([0,T]; L^2)$ such that
\begin{equation}
	i\partial_0 \WPHI = H\WPHI + mF_\lambda(\phi)
	\quad\text{for $t\in[0,T]$,}\quad
	\phi(0) = \phi_0.
\label{E:IVP}
\end{equation}
If such a solution exists, then it satisfies
\[
	\WPHI(t) = \UO(t)\WPHI_0 + m \int_0^t \UO(t-s) F_\lambda\big( \phi(s) \big) \,ds
	\quad\text{for $t\in[0,T]$,}
\]
with $\{\UO(t)\}_{t\GS 0}$ the strongly continuous one-parameter unitary group on $L^2$ generated by the Hamiltonian $H$; see Lemma~4.1.1 in \cite{CazenaveHarauxMartel1998}. The existence of such a one-parameter group is well-known to be equivalent to self-adjointness of the operator $H$, because of Stone's theorem. For self-adjointness of $H$ we need suitable assumptions on $A_0, \BA$. For simplicity, here we only consider the case of a time-\emph{independent} electromagnetic potential satisfying the following conditions: for some $p>3$ we have
\[
	\BA \in L^p(\R^3) + L^\infty(\R^3)
	\quad\text{and}\quad
	\text{$|qA_0| \LS m$ a.e.}
\]
Then self-adjointness of $H$ follows from Theorem~3.3.7 in \cite{Tretter2008}; see also Remark~3.3.8 there. More general conditions will be considered elsewhere.

Existence of solutions to the inhomogeneous initial value problem \eqref{E:IVP} now follows from standard arguments from the theory of semilinear partial differential equations. Since the map $\phi \mapsto F_\lambda(\phi)$ is Lipschitz continuous on $L^2$ with an explicit bound on the Lipschitz constant (see Lemma~\ref{lem:nonlinearity_properties}), local well-posedness of \eqref{E:IVP} follows from Theorem~4.3.4 in \cite{CazenaveHarauxMartel1998}. Solutions exist as long as they remain finite in size. In our case, this is guaranteed for all finite times in both $L^2$ and $H^1$ because of items~(1), (2) in Lemma~\ref{lem:nonlinearity_properties} and Gronwall's lemma. Indeed the $L^2$- and $H^1$-norms of the solution $\phi$ do not grow faster than $\exp(2mt)$ and $\exp(cm(\lambda^{-1}+1)t)$, respectively, with $c>0$ some constant. We refer the reader to Section~4 of \cite{CazenaveHarauxMartel1998} for details. Note that for $\lambda\to 0$ the $H^1$-control of the solution degenerates, so for the non-regularized Dirac equation a suitable weak solution concept must be utilized.


\section{Hydrodynamics formulation}

In this section, we will derive a hydrodynamics formulation of Daviau's improved Dirac equation. There will be two separate parts to it, corresponding to the left and right spinor waves, that interact with each other via the Dirac current. For each part there will be two contributions, a chiral current and an energy-momentum tensor. Therefore the hydrodynamics formulation of the Dirac equation is rather complex. Note that the existence of the chiral current follows from Noether's theorem because of the invariance of the Dirac Lagrangian under gauge transformations, whereas the energy-momentum tensor is related to the invariance under spacetime translations. We will assume throughout that all functions are sufficiently smooth so that the following computations are justified. In this sense, the derivation is somewhat formal. A mathematically rigorous examination of the resulting model will be cconsidered elsewhere. We refer the reader to \cite{BialynickiBirula1995} for a similar result for the Weyl equation.


\subsection{Hydrodynamic equations}

We split the Dirac equation into left and right contributions, by multiplying \eqref{E:DIRACEQ} from the right by the operator $P = \frac{1}{2}(1+\FE_3)$ and its spatial reversal $\OP$, respectively, which project $\CL{3}$ into the corresponding minimal left-sided ideals; see Section~\ref{SS:LSI}. We obtain an equation
\begin{equation}
\tcboxmath{ 
	\nabla(\WPHI P) + i\Big( qA + mV \Big) (\WPHI P) = 0
} 
\label{E:SIMPL}
\end{equation}
for the spinor function $\WPHI P$ because $\FE_3 P = P$. For $\WPHI\OP$ we obtain a similar equation with a minus sign in the second term because $\FE_3\OP = -\OP$. We will write
\begin{equation}
	\Gamma \coloneq \Gamma_\kappa \FE^\kappa
	\quad\text{with}\quad
	\text{$\Gamma_\kappa \coloneq qA_\kappa + mV_\kappa$ for $\kappa = 0\ldots 3$.}
\label{E:DEFGA}
\end{equation}

\begin{lemma}
\label{L:MOME}
Suppose that $\phi$ is a solution of \eqref{E:DIRACEQ}. Recalling \eqref{E:DEFGA}, we define the chiral current $\FJ \coloneq \phi\OP\phi^\dagger$ and Tétrode's energy-momentum tensor $T$ by
\begin{equation}
	{T_\nu}^\mu
		\coloneq \RE\Big( P\OPHI \FE^\mu (-i\partial_\nu + \Gamma_\nu)\WPHI P \Big)_0
	\quad\text{for $\mu, \nu = 0\ldots 3$.}
\label{E:TETRODE}
\end{equation}
With tensor field $G \coloneq G_{\mu\nu} \FE^\mu \FE^\nu$ and $G_{\mu\nu} \coloneq \partial_\mu \Gamma_\nu - \partial_\nu \Gamma_\mu$, there holds
\[
\tcboxmath{ 
	\partial_\mu {T_\nu}^\mu = G_{\mu\nu} j^\mu
	\quad\text{for $\nu = 0\ldots 3$}
} 
\]
\end{lemma}

Note that $G$ contains the electromagnetic tensor $F_{\mu\nu} = \partial_\mu A_\nu - \partial_\nu A_\mu$.


\begin{proof}
We observe first that for $\mu, \nu = 0\ldots 3$
\[
	{T_\nu}^\mu = \frac{1}{2} \bigg( \Big( P\OPHI \FE^\mu (-i\partial_\nu + \Gamma_\nu)\WPHI P \Big)
		+ \Big( P\OPHI \FE^\mu (-i\partial_\nu + \Gamma_\nu)\WPHI P \Big)^\dagger \bigg)_0.
\]
We consider the two parts separately. We have
\begin{equation}
	\Big( P\OPHI \FE^\mu (-i\partial_\nu + \Gamma_\nu)\WPHI P \Big)
	= -i \Big( P\OPHI \FE^\mu (\partial_\nu\WPHI)P \Big)
	+ \Gamma_\nu \Big( P\OPHI \FE^\mu \WPHI P \Big).
\label{E:FIRR}
\end{equation}
Because of Lemma~\ref{L:CUU} and the definition of the chiral current, the scalar part of the last term on the right-hand side of \eqref{E:FIRR} is $\Gamma_\nu j^\mu$. Using \eqref{E:FIRR}, we obtain
\[
	\Big( P\OPHI \FE^\mu (-i\partial_\nu + \Gamma_\nu)\WPHI P \Big)^\dagger
	= i \Big( P(\partial_\nu\OPHI) \FE^\mu \WPHI P \Big)
	+ \Gamma_\nu \Big( P\OPHI \FE^\mu \WPHI P \Big).
\]
Recall that $\FE^\mu$ and $P$ are real, and so is $\Gamma_\nu$. Combining terms, we arrive at
\begin{equation}
	{T_\nu}^\mu = -\frac{i}{2} \bigg( \FE^\mu \Big(
		(\partial_\nu\WPHI)P\OPHI - \WPHI P(\partial_\nu\OPHI) \Big) \bigg)_0
	+ \Gamma_\nu j^\mu
\label{E:NULL}
\end{equation}
 (cf.\ (1.284)/(1.285) in \cite{DaviauBertrandSocrounGirardot2022}), where we have used \eqref{E:ROTAT} to rotate arguments cyclically in the scalar part, and $P^2 = P$. We now compute the spacetime divergence
\[
	\partial_\mu {T_\nu}^\mu
	= -\frac{i}{2} \bigg( \FE^\mu \partial_\mu \Big(
		(\partial_\nu\WPHI)P\OPHI - \WPHI P(\partial_\nu\OPHI) \Big) \bigg)_0
	+ \partial_\mu (\Gamma_\nu j^\mu).
\]
For the last term on the right-hand side, we observe that
\[
	\partial_\mu (\Gamma_\nu j^\mu)
	= (\partial_\mu \Gamma_\nu) j^\mu + \Gamma_\nu (\partial_\mu j^\mu)
	= (\partial_\mu \Gamma_\nu) j^\mu
\]
because the chiral current $j$ is conserved; see Lemma~\ref{L:CURR} below. Next we get
\[
	\bigg( \FE^\mu \partial_\mu \Big(
		(\partial_\nu\WPHI)P\OPHI \Big) \bigg)_0
	= \bigg( \Big( \partial_\nu (\nabla\WPHI) \Big)P\OPHI
		+ (\partial_\nu\WPHI)P (\nabla\WPHI)^\dagger \bigg)_0,
\]
rotating arguments cyclically in the scalar part and using that
\begin{equation}
	(\partial_\mu\OPHI) \FE^\mu
	= \Big( \FE^\mu (\partial_\mu\WPHI) \Big)^\dagger
	= (\nabla\WPHI)^\dagger.
\label{E:SWITCH}
\end{equation}
Recall that $\FE^\mu \partial_\mu = \nabla$ and $\nabla \partial_\nu = \partial_\nu \nabla$. Since $\WPHI P$ satisfies \eqref{E:SIMPL}, we obtain
\begin{align*}
	\Big( \partial_\nu (\nabla\WPHI) \Big) P
	& = -i \Big( (\partial_\nu\Gamma) \WPHI + \Gamma (\partial_\nu\WPHI) \Big) P,
\\
	P (\nabla\WPHI)^\dagger
	& = i P \OPHI \Gamma
\end{align*}
because $\Gamma$ is real. Combining terms, we find
\begin{align*}
	\bigg( \FE^\mu \partial_\mu \Big(
		(\partial_\nu\WPHI)P\OPHI \Big) \bigg)_0
	& = -i \bigg( \Big( (\partial_\nu\Gamma) \WPHI + \Gamma (\partial_\nu\WPHI) \Big) P\OPHI
		- (\partial_\nu\WPHI) P\OPHI \Gamma \bigg)_0
\\
	& = -i \Big( (\partial_\nu\Gamma) \WPHI P\OPHI \Big)_0,
\end{align*}
after rotating arguments and cancellation. Analogously, we obtain
\begin{align*}
	\bigg( \FE^\mu \partial_\mu \Big(
		\WPHI P(\partial_\nu\OPHI) \Big) \bigg)_0
	& = -i \bigg( \Gamma \WPHI P (\partial_\nu\OPHI)
		- \WPHI P \Big( (\partial_\nu\OPHI) \Gamma + \OPHI \big( \partial_\nu\Gamma \big) \Big) \bigg)_0
\\
	& = i \Big( (\partial_\nu\Gamma) \WPHI P\OPHI \Big)_0.
\end{align*}
Writing $\Gamma = \Gamma_\kappa \FE^\kappa$ and collecting terms, we observe that
\[
	\partial_\mu {T_\nu}^\mu
	= -(\partial_\nu\Gamma_\kappa) \Big( \FE^\kappa (\WPHI P\OPHI) \Big)_0
	+ (\partial_\mu \Gamma_\nu) j^\mu.
\]
The result now follows from the definition of $\FJ$ and Lemma~\ref{L:CUU}.
\end{proof}

\begin{remark}
The trace of the energy-momentum tensor is given by
\[
	{T_\mu}^\mu = -\frac{i}{2} \bigg( \FE^\mu \Big(
		(\partial_\mu\WPHI)P\OPHI - \WPHI P(\partial_\mu\OPHI) \Big) \bigg)_0
	+ \Gamma_\mu j^\mu;
\]
see \eqref{E:NULL}. Using \eqref{E:SWITCH}, \eqref{E:SIMPL}, and the definition of the chiral current, we obtain
\[
	{T_\mu}^\mu
	= -\frac{i}{2} \Big( (\nabla\WPHI) P\OPHI - \WPHI P (\nabla\WPHI)^\dagger \Big)_0
	+ \Gamma_\mu j^\mu
	= -\Gamma_\kappa \Big( \FE^\kappa \WPHI P\OPHI \Big)_0
	+ \Gamma_\mu j^\mu
	= 0.
\]
\end{remark}

\begin{lemma}
\label{L:CURR}
Let $\phi$, $\FJ$, and $T$ be as in Lemma~\ref{L:MOME}. Then $\partial_\mu j^\mu = 0$ and
\begin{equation}
\tcboxmath{ 
	\partial_0 j^k + \partial_k j^0 + 2 {\varepsilon^{kl}}_n {T_l}^n = 0
	\quad\text{for $k = 1\ldots 3$.}
} 
\label{E:CURRENT}
\end{equation}
\end{lemma}

\begin{proof}
Using \eqref{E:SIMPL}, $\BGAMMA \coloneq q\BA + m \BVV$, and $\nabla = \partial_0 + \BNABLA$ (see \eqref{E:NADEC}), we have
\begin{equation}
\begin{aligned}
	\partial_0 (\WPHI P\OPHI)
		& = \Big( \partial_0(\WPHI P) \Big) P\OPHI + \WPHI P \Big( \partial_0(\WPHI P)^\dagger \Big)
			= \Big( (-i \Gamma_0 \WPHI) P \OPHI + \WPHI P (-i \Gamma_0 \WPHI)^\dagger \Big)
\\
		& + \Big( (-\BNABLA \WPHI) P\OPHI + \WPHI P(-\BNABLA \WPHI)^\dagger \Big)
			+ \Big( (-i \BGAMMA \WPHI) P\OPHI + \WPHI P (-i \BGAMMA \WPHI)^\dagger \Big).
\end{aligned}
\label{E:RHOT}
\end{equation}
The first term on the right-hand side of \eqref{E:RHOT} vanishes.

Our goal is to extract differential equations for the components $j^\kappa$ of the chiral current. We multiply \eqref{E:RHOT} from the left by $\FE^\kappa = \overline{\FE_\kappa}$ and compute the scalar part. For the last two terms on the right-hand side of \eqref{E:RHOT} we have
\begin{align}
	\bigg( \FE^\kappa \Big( (\BNABLA\WPHI)P\OPHI
		+ \WPHI P(\BNABLA\WPHI)^\dagger \Big) \bigg)_0
	& = \bigg( \FE^\kappa\FE^l \Big( (\partial_l\WPHI)P\OPHI \Big)
		+ \FE^l\FE^\kappa \Big( \WPHI P(\partial_l\OPHI) \Big) \bigg)_0,
\label{E:UNNO}
\\
  \bigg( \FE^\kappa \Big( (i \BGAMMA \WPHI) P\OPHI + \WPHI P (i \BGAMMA \WPHI)^\dagger \Big) \bigg)_0
	& = i \Gamma_l \Big( (\FE^\kappa\FE^l - \FE^l\FE^\kappa) \WPHI P\OPHI \Big)_0;
\label{E:SECTERM}
\end{align}
recall \eqref{E:SWITCH}. For $\kappa, \nu = 0 \ldots 3$ we have the identity (see \eqref{E:LIFT})
\begin{equation}
	\Big( \FE^\kappa \partial_\nu (\WPHI P\OPHI) \Big)_0
	= \partial_\nu \Big( \FE^\kappa (\WPHI P\OPHI) \Big)_0
	= \partial_\nu j^\kappa.
\label{E:ALLES}
\end{equation}

Let us consider the case $\kappa = 0$ first. Then \eqref{E:UNNO}/\eqref{E:ALLES} gives
\[
	\Big( (\BNABLA\WPHI)P\OPHI + \WPHI P(\BNABLA\WPHI)^\dagger \Big)_0
		= \bigg( \FE^l \Big( (\partial_l\WPHI)P\OPHI \Big) + \FE^l \Big( \WPHI P(\partial_l\OPHI) \Big) \bigg)_0
		= \partial_l j^l,
\]
while \eqref{E:SECTERM} vanishes. From this and \eqref{E:RHOT} we obtain $\partial_\mu j^\mu = 0$.

We restrict $\kappa$ to $k = 1\ldots 3$ and use \eqref{E:ANTICO2} and Remark~\ref{R:SFG} to rewrite \eqref{E:UNNO} as
\begin{align*}
	\bigg( \FE^k \Big( (\BNABLA\WPHI)P\OPHI
		+ \WPHI P(\BNABLA\WPHI)^\dagger \Big) \bigg)_0
	& = \delta^{kl} \, \Big( (\partial_l\WPHI)P\OPHI
		+ \WPHI P(\partial_l\OPHI) \Big)_0
\\
	& - i {\varepsilon^{kl}}_n \,
		\bigg( \FE^n \Big( (\partial_l\WPHI)P\OPHI
			- \WPHI P(\partial_l\OPHI) \Big) \bigg)_0.
\end{align*}
The first term on the right-hand side equals $\delta^{kl} (\partial_l j^0)$; see \eqref{E:ALLES}. For the second one, we use expression \eqref{E:NULL} for the energy-momentum tensor to obtain
\[
	\bigg( \FE^n \Big( (\partial_l\WPHI)P\OPHI
		- \WPHI P(\partial_l\OPHI) \Big) \bigg)_0
	= 2i ({T_l}^n - \Gamma_l j^n).
\]
Using \eqref{E:ANTICO2} once more, we derive from \eqref{E:SECTERM} that
\[
    \bigg( \FE^k \Big( (i \BGAMMA \WPHI) P\OPHI + \WPHI P (i \BGAMMA \WPHI)^\dagger \Big) \bigg)_0
		= i \Gamma_l \Big( (\FE^k \FE^l - \FE^l \FE^k) \WPHI P \OPHI \Big)_0
		= 2 {\varepsilon^{kl}}_n \Gamma_l j^n.
\]
Using \eqref{E:RHOT}/\eqref{E:ALLES} and combining terms, we obtain for $k = 1\ldots 3$ that
\[
	\partial_0 j^k
		= -\delta^{kl} \, \partial_l j^0
		- 2 {\varepsilon^{kl}}_n ({T_l}^n - \Gamma_l j^n)
		- 2 {\varepsilon^{kl}}_n \Gamma_l j^n.
\]
Cancelling and rearranging terms, we obtain equation \eqref{E:CURRENT}.
\end{proof}

\begin{lemma}
\label{L:EXPR}
Let $\phi$, $\FJ$, and $T$ be as in Lemma~\ref{L:MOME}. Define
\begin{equation}
	\text{$\RHO \coloneq j_0$, $v_\nu \coloneq j_\nu/j_0$, and $p_\nu \coloneq {T_\nu}^0$}
	\quad\text{for $\nu=0\ldots 3$.}
\label{E:NOTAT}
\end{equation}
Then $j_\mu j^\mu = 0$ (implying that $\|\vec{v}\| = 1$) and
\begin{equation}
\tcboxmath{ 
	{T_n}^l = p_n v^l - \frac{1}{2}\RHO {\varepsilon^{lk}}_o v_k (\partial_n v^o)
	\quad\text{for $l,n = 1\ldots 3$.}
} 
\label{E:RESULT}
\end{equation}
\end{lemma}

\begin{proof}
From definition of the chiral current in Lemma~\ref{L:MOME}, we obtain
\[
	\FJ\overline{\FJ}
		= \phi\OP\phi^\dagger \, \overline{\phi\OP\phi^\dagger}
		= \phi\OP (\phi^\dagger\widehat{\phi}) P\OPHI
		= \phi\OP (\OPHI\phi)^\dagger P\OPHI
		= (\OPHI\phi)^\dagger \, \phi\OP P\OPHI
		= 0
\]
because $\OPHI\phi$ is always a scalar and $\OP P = 0$. It follows that $\FJ$ is lightlike, thus $j_\mu j^\mu = 0$. Writing $\FJ = \RHO(1 + \BV)$ with $\BV = v^k \FE_k$, we obtain the claim $\|\vec{v}\| = 1$ ($v_k v^k = -1$).

We multiply the expansion $\FP = \FP_\mu \FE^\mu$ with $\FP = \WPHI P\OPHI$ by $P\OPHI$ from the left and by $\FE^\lambda (\partial_\nu\WPHI)P$ from the right (recall that $P^2 = P$) to obtain
\begin{equation}
	(P\OPHI\WPHI P) \Big( P\OPHI\FE^\lambda (\partial_\nu\WPHI)P \Big)
		= (\WPHI P\OPHI)_\mu \Big( P\OPHI\FE^\mu\FE^\lambda(\partial_\nu\WPHI)P \Big).
\label{E:OPP}
\end{equation}
Note that because of \eqref{E:EXPA2} and since $P^2=P$, terms on the left-hand side of \eqref{E:OPP} are products of elements in the one-dimensional double-sided ideal discussed at the end of Section~\ref{SS:LSI}, specifically of elements of the form $\alpha P$, with $\alpha \in \C$ some number. For such elements, multiplication is commutative and we can change the order of factors. By Lemma~\ref{L:CUU}, the coefficients of $\FJ$ can be computed as
\[
	j^\mu
		= (\phi\OP\phi^\dagger)^\mu
		= (\WPHI P\OPHI)_\mu
		= \Big( P \OPHI\FE^\mu\WPHI P \Big)_0
	\quad\text{for $\mu=0\ldots 3$.}
\]
Combining this with \eqref{E:EXPA2}, we conclude that (notice the factor $2$)
\begin{equation}
	P\OPHI\WPHI P
		= P (P\OPHI\WPHI P) P
		= 2 (P\OPHI\WPHI P)_0 P
		= 2 (\WPHI P\OPHI)_0 P
		= 2 j^0 P
\label{E:RESOLVE}
\end{equation}
because $P^2 = P$. The term with $\mu=0$ on the right-hand side of \eqref{E:OPP} is therefore the same as the one on the left and can be canceled once. We arrive at
\begin{equation}
	j_0 \Big( P\OPHI\FE^\lambda (\partial_\nu\WPHI)P \Big)
	= -j_k \Big( P\OPHI\FE^k\FE^\lambda (\partial_\nu\WPHI)P \Big),
\label{E:OP}
\end{equation}
where we have used that $j_k = -j^k$ for $k=1\ldots 3$ and $j_0 = j^0$; see \eqref{E:DUALB}.

Similarly, we multiply $\FP = \FP_\mu \FE^\mu$ with $\FP = \WPHI P\OPHI$ by $P(\partial_\nu\OPHI)\FE^\lambda$ and $\WPHI P$ from the left and right, respectively. After simplification, we obtain the equality
\begin{equation}
	j_0 \Big( P(\partial_\nu\OPHI)\FE^\lambda\WPHI P \Big)
	= -j_k \Big( P(\partial_\nu\OPHI)\FE^\lambda\FE^k\WPHI P \Big).
\label{E:OP2}
\end{equation}
Subtracting \eqref{E:OP2} from \eqref{E:OP} and computing the scalar parts, we find
\[
	j_0 \bigg( \FE^\lambda \Big( (\partial_\nu\WPHI)P\OPHI - \WPHI P(\partial_\nu\OPHI) \Big) \bigg)_0
	= -j_k \bigg(
		\FE^k\FE^\lambda \Big( (\partial_\nu\WPHI)P\OPHI \Big)
		- \FE^\lambda\FE^k \Big( \WPHI P(\partial_\nu\OPHI) \Big) \bigg)_0.
\]
For $\lambda = 0$ we can use \eqref{E:NULL} to rewrite this identity in the form
\[
	j_0 \, 2i ({T_\nu}^0 - \Gamma_\nu j^0) + j_k \, 2i ({T_\nu}^k - \Gamma_\nu j^k) = 0,
\]
which we rewrite as
\[
	j_0 {T_\nu}^0 + j_k {T_\nu}^k = \Gamma_\nu ( j_0 j^0 + j_k j^k) = 0
\]
because $j_\mu j^\mu = 0$. We conclude that $-p_n = v_k {T_n}^k$ for $n = 0\ldots 3$.

Restricting $\lambda, \nu$ to $l,n = 1\ldots 3$, we use \eqref{E:ANTICO2} to obtain
\begin{align*}
	& j_0 \bigg( \FE^l \Big( (\partial_n \WPHI)P\OPHI - \WPHI P(\partial_n \OPHI) \Big) \bigg)_0
\\
	& \quad = -j_k
		\delta^{kl} \Big( (\partial_n \WPHI)P\OPHI - \WPHI P(\partial_n \OPHI) \Big)_0
		+ i{\varepsilon^{kl}}_o j_k \bigg( \FE^o \Big( (\partial_n \WPHI)P\OPHI
			+ \WPHI P(\partial_n \OPHI) \Big) \bigg)_0.
\end{align*}
Since $-j_k \delta^{kl} = j^l$ this can be rewritten using \eqref{E:NULL} as
\[
	j_0 \, 2i ({T_n}^l - \Gamma_n j^l)
		= j^l \, 2i ({T_n}^0 - \Gamma_n j^0)
		+ i{\varepsilon^{kl}}_o j_k (\partial_n j^o).
\]
Here terms with $\Gamma_n$ cancel because $j_0 = j^0$.
With $j^o = \RHO v^o$  etc.\ we can simplify
\[
	{\varepsilon^{kl}}_o v_k (\partial_n j^o)
		= (\partial_n \RHO) {\varepsilon^{kl}}_o v_k v^o
			+ \RHO {\varepsilon^{kl}}_o v_k (\partial_n v^o)
		= -\RHO {\varepsilon^{lk}}_o v_k (\partial_n v^o)
\]
because of antisymmetry of the Levi-Civita tensor. Then \eqref{E:RESULT} follows.
\end{proof}

We can now state our hydrodynamics formulation of the Dirac equation.

\begin{theorem}
Suppose $\phi$ is a solution of the Dirac equation \eqref{E:DIRACEQ}. Let $(\RHO_\cdot,\BV_\cdot,\BP_\cdot)$ for $\cdot = R,L$ be defined as in Lemma~\ref{L:EXPR} using $P$ and $\OP$, respectively. Then
\begin{equation}
\tcboxmath{ 
\begin{aligned}
	\vphantom{\frac{1}{2}}
	\partial_0 \RHO + \nabla\cdot(\RHO\BV)
	&= 0
\\
	\partial_0 (\RHO\BV)
		+ \nabla \cdot (\RHO\BV\otimes\BV)
	& = \big( \BV\times(\BV\times\nabla\RHO) \big) + 2 \Big( \BP + \RHO(\nabla \times \BV) \Big)  \times \BV
\\
	\partial_0 p_n + \nabla \cdot (p_n \BV)
	& = \frac{1}{2} \nabla \cdot \big( \RHO\BV \times (\partial_n\BV) \big)
		\pm \RHO G_{n\kappa} v^\kappa
	\quad\text{for $n = 1\ldots 3$}
\end{aligned}
} 
\label{E:HYD}
\end{equation}
with sign $+$ for the case $(\RHO,\BV,\BP) = (\RHO_R,\BV_R,\BP_R)$ and sign $-$ otherwise.
\end{theorem}

\begin{proof}
The continuity equation in \eqref{E:HYD} follows from Lemma~\ref{L:CURR}. Then we use Lemma~\ref{L:EXPR} to replace ${T_l}^n$ on the left-hand side of \eqref{E:CURRENT}. We write
\[
	2 {\varepsilon^{kl}}_n {T_l}^n
		= 2 {\varepsilon^{kl}}_n p_l v^n
		- \RHO {\varepsilon^{kl}}_n {\varepsilon^{nm}}_o v_m (\partial_l v^o).
\]
Recalling Remark~\ref{R:SFG}, we observe that
\[
	{\varepsilon^{kl}}_n p_l v^n
		= -\varepsilon_{kln} p^l v^n
		= -(\BP\times\BV)^k.
\]
Indeed notice that by standard convention $(a \times b)^i = \varepsilon_{ijk} a^j b^k$ for vectors $a, b \in \R^3$.
Similarly, using the properties of the Levi-Civita symbol, we obtain
\begin{align*}
	{\varepsilon^{kl}}_n {\varepsilon^{nm}}_o v_m (\partial_l v^o)
		& = \sum_{l,m,o} \varepsilon_{kln} \varepsilon^{mon} v^m (\partial_l v^o)
\\
		& = -\sum_{l,m,o} \big( \delta_k^m \delta_l^o
			- \delta_k^o \delta_l^m \big) v^m (\partial_l v^o)
		= (\BV\cdot\nabla) v^k - v^k (\nabla\cdot\BV).
\end{align*}
Note that according to Remark~\ref{R:SFG}, we have $\varepsilon^{ijk} = -\varepsilon_{ijk}$ (lowering three indices) whereas in text books there is usually no minus sign. This amounts for the additional minus after the second equality. We obtain the identity
\[
	2 {\varepsilon^{kl}}_n {T_l}^n
		= -2 (\BP\times\BV)^k
			- \RHO \big( (\BV\cdot\nabla) v^k - v^k (\nabla\cdot\BV) \big).
\]
Since $a \times (b \times c) = b(a \cdot c) - c(a\cdot b)$ for $a, b, c \in \R^3$, it follows that
\begin{gather*}
	\big( \BV \times (\nabla\times \BV) \big)^k
		= \BV\cdot (\partial_k\BV) - (\BV\cdot \nabla) v^k
		= -(\BV\cdot \nabla) v^k,
\\
 	\big( \BV \times (\BV\times \nabla\RHO) \big)^k
		= v^k \big((\BV\cdot\nabla)\RHO \big) - (\partial_k\RHO) (\BV\cdot\BV)
		= v^k \big((\BV\cdot\nabla)\RHO \big) - \partial_k\RHO.
\end{gather*}
for $k=1\ldots 3$. Here we have used that $\BV\cdot\BV = 1$ and $\BV\cdot\partial_k\BV = 0$. We also need
\[
	-\RHO \big( (\BV\cdot\nabla) v^k \big) + \RHO v^k (\nabla\cdot\BV)
		= \nabla\cdot(\RHO v^k \BV)
			-2\RHO \big( (\BV\cdot\nabla) v^k \big)
			-v^k \big( (\BV\cdot\nabla)\RHO \big).
\]
Combining all terms, we find (with $j_0 = \RHO$) that
\begin{align*}
	& \partial_k j^0 + 2 {\varepsilon^{kl}}_n {T_l}^n
\\
	&\qquad
		= -2(\BP\times\BV)^k
			+ \nabla\cdot(\RHO v^k\BV)
			- 2\RHO\big( (\nabla\times\BV)\times\BV \big)^k
			- \big( \BV\times(\BV\times\nabla\RHO) \big)^k.
\end{align*}
This gives the second equation in \eqref{E:HYD}. For the third one, we combine Lemmas~\ref{L:MOME} and \ref{L:EXPR}. This proves the result for the case of the right spinor.

For the left spinor, we simply notice that substituting $\OP$ for $P$ results in a change of sign in the zeroth order term of the Dirac equation \eqref{E:DIRACEQ} because $\FE^3 \OP = -\OP$. The only place where this term appears in the hydrodynamic system is on the right-hand side of the conservation law for the energy-momentum tensor in Lemma~\ref{L:MOME}. Consequently, we obtain a sign change in the third equation of \eqref{E:HYD}.
\end{proof}

\begin{remark}
Consider the right spinor case. Since $\RHO = j^0 = (\phi\OP\phi^\dagger)_0$ the spinor $\phi P$ vanishes if and only if $\RHO$ is zero. The energy-momentum \eqref{E:TETRODE} is thus absolutely continuous with respect to $\RHO$ so there exists a vector field $\BU$ such that $\BP = \RHO \BU$. Note that the convective derivatives in \eqref{E:HYD} all involve the same transport velocity $\BV$. Using the product rule, we can therefore rewrite \eqref{E:HYD} in the form
\begin{equation}
\begin{aligned}
	\vphantom{\frac{1}{2}}
	\dot{\RHO} &= -\RHO (\nabla\cdot\BV)
\\
	\dot{\BV}
		& = \BV\times\big( \BV\times\nabla\log(\RHO) \big)
			+ 2 \big( \BU + \nabla\times\BV  \big) \times \BV
\\
	\dot{u}_n
		& = \frac{1}{2\RHO} \nabla \cdot \big( \RHO\BV \times (\partial_n\BV) \big)
			\pm G_{n\kappa} v^\kappa
	\quad\text{for $n=1\ldots 3$}
\end{aligned}
\label{E:FLOWL}
\end{equation}
where the dot stands for the directional derivative $\partial_0 + \BV\cdot\nabla$ in the direction of the flowlines, which are the solutions of the ordinary differential equation $\dot{\gamma} = \BV(\cdot,\gamma)$. These are different from the flowlines determined by the Dirac current $J = NV$, which are the solutions of $\dot{\gamma} = V(\cdot, \gamma)$; see \eqref{E:PILOT}. Indeed the picture suggested by \eqref{E:FLOWL} is the following: For any solution $\phi$ of \eqref{E:DIRACEQ}, the Dirac current defines a congruence of flowlines. These flowlines guide the motion of the left and right spinors in the sense that their respective flowlines wind around the flowlines of the Dirac current. Because of this, we call the Dirac current the pilot wave, using De Broglie's terminology. The left and right spinor flowlines move with light speed as $\|\BV\| = 1$, whereas the velocity of the Dirac current is subluminal.

This is the picture discussed by Hestenes \cites{Hestenes2020a,Hestenes2020b} who offers such a helical motion as a possible explanation for Schrödinger's Zitterbewegung. The evolution \eqref{E:FLOWL} is also reminiscent of the classical model of the Dirac electron proposed by Barut \& Zanghi \cite{BarutZanghi1984}; see in particular equation (2) there. Their model also considers separately the velocity $\BV$ of the particle whose time derivative (along flowlines) is influenced by the momentum $\BP$. The time derivative of the momentum is given by the Lorentz force. Since the Barut-Zanghi model describes only a single particle, there is no interaction between left/right spinors and the Dirac current.

Notice that if the momentum $\BP$ in \eqref{E:HYD} were collinear with the velocity $\BV$ as in classical mechanics, then there would be no force from $\BP$ acting on $\BV$ because the cross product of two collinear vectors vanish. The same remark applies to the velocity $V$ of the Dirac current. The right-hand side of the third equation in \eqref{E:HYD} thus captures the quantum effects, in a similar way as the quantum potential \eqref{E:BOHM} captures the quantum effects in the hydrodynamics model \eqref{E:JJH} of the Schrödinger equation. In \eqref{E:HYD} the particle mass $m$ only enters as a constant determining the strength of the coupling of the left/right currents to the Dirac current.

\end{remark}


\subsection{Quantization condition}

The hydrodynamical system \eqref{E:HYD} could admit solutions that are not derived from the nonlinear Dirac equation. Dirac solutions must have additional properties; they must satisfy a so-called quantization condition, which amounts to the curl of the momentum field $\BU = \BP/\RHO$ being predetermined in terms of derivatives of the velocity $\BV$. In this case, the number of unknowns of \eqref{E:HYD} reduces to two sets of four real-valued fields each for the left/right chiral parts: the density $\RHO$, the three components of the velocity $\BV$, and the longitudinal component $\BP\cdot\BV$ of the momentum $\BP = \RHO \BU$. Recall that $\|\BV\| = 1$, i.e., characteristics move with the speed of light (like Weyl fermions). This is consistent with the nonlinear Dirac equation because spinors have four complex components.

\begin{proposition}
Let $v^l, p_n \eqcolon \RHO u_n$ be defined as in Lemma~\ref{L:EXPR}. Then
\begin{equation}
	\varepsilon^{imn} \bigg( \partial_m (u_n - \Gamma_n)
		+ \frac{1}{4} \BV \cdot \Big( (\partial_m\BV) \times (\partial_n \BV) \Big) \bigg) = 0
\quad\text{for $i=1\ldots 3$.}
\label{E:QUCO}
\end{equation}
\end{proposition}

\begin{proof} We proceed in four steps.
\medskip

\textbf{Step~1.} Differentiating the expansion $\FP = \FP_\kappa \FE^\kappa$ with $\FP = \WPHI P \OPHI$, we get
\begin{equation}
	(\partial_\mu \WPHI) P \OPHI + \WPHI P (\partial_\mu \OPHI)
		= \partial_\mu(\WPHI P \OPHI)_\kappa \FE^\kappa
		= \partial_\mu(\WPHI P \OPHI)_0
		+ \partial_\mu(\WPHI P \OPHI)_k \FE^k.
\label{E:CURL2}
\end{equation}
We multiply \eqref{E:CURL2} from the left by $P \OPHI$, by $\FE^\lambda (\partial_\nu \WPHI) P$ from the right. Then
\begin{equation}
\begin{aligned}
	& \Big( P\OPHI (\partial_\mu\WPHI) P \Big) P\OPHI \FE^\lambda (\partial_\nu \WPHI) P
		+ \Big( P\OPHI \WPHI P \Big) P (\partial_\mu\OPHI) \FE^\lambda (\partial_\nu \WPHI) P
\\
	& \qquad
		= \Big( \partial_\mu (\WPHI P\OPHI)_0 \Big) P\OPHI \FE^\lambda (\partial_\nu\WPHI) P
			+ \Big( \partial_\mu (\WPHI P\OPHI)_k \Big) P\OPHI \FE^k\FE^\lambda (\partial_\nu\WPHI) P.
\end{aligned}
\label{E:USCH}
\end{equation}
We have used that $P^2=P$. Note that the parentheses on the right-hand side are scalar, hence commute with $P\OPHI$. Differentiating \eqref{E:RESOLVE}, we obtain the identity
\[
 	\Big( \partial_\mu (\WPHI P \OPHI)_0 \Big) P\OPHI \FE^\lambda (\partial_\nu\WPHI) P
	 = \frac{1}{2} \Big( P \OPHI (\partial_\mu\WPHI) P + P (\partial_\mu\OPHI) \WPHI P \Big)
		P\OPHI \FE^\lambda (\partial_\nu\WPHI) P,
\]
which we can absorb into the first term on the left-hand side of \eqref{E:USCH}. Then
\begin{equation}
\begin{aligned}
	& \frac{1}{2} \Big( P\OPHI (\partial_\mu\WPHI)P - P(\partial_\mu \OPHI) \WPHI P \Big)
		P\OPHI \FE^\lambda (\partial_\nu \WPHI) P
\\
	& \qquad
			+ \Big( P\OPHI \WPHI P \Big) P (\partial_\mu\OPHI) \FE^\lambda (\partial_\nu \WPHI) P
		= \Big( \partial_\mu (\WPHI P\OPHI)_k \Big) P\OPHI \FE^k\FE^\lambda (\partial_\nu\WPHI) P.
\end{aligned}
\label{E:IDONE}
\end{equation}
Multiplying \eqref{E:CURL2} by $\WPHI P$ from the right, by $P (\partial_\nu \OPHI) \FE^\lambda$ from the left, we get
\begin{equation}
\begin{aligned}
	& P (\partial_\nu\OPHI) \FE^\lambda (\partial_\mu\WPHI) P \Big( P\OPHI \WPHI P \Big)
		+ P (\partial_\nu\OPHI) \FE^\lambda \WPHI P \Big( P(\partial_\mu\OPHI) \WPHI P \Big)
\\
	& \qquad
		= \Big( \partial_\mu (\WPHI P\OPHI)_0 \Big) P (\partial_\nu\OPHI) \FE^\lambda \WPHI P
			+ \Big( \partial_\mu (\WPHI P\OPHI)_k \Big) P (\partial_\nu\OPHI) \FE^\lambda \FE^k \WPHI P.
\end{aligned}
\label{E:LLO}
\end{equation}
Arguing as above and reordering, we can rewrite \eqref{E:LLO} in the form
\begin{equation}
\begin{aligned}
	& -\frac{1}{2} \Big( P\OPHI (\partial_\mu\WPHI)P - P(\partial_\mu \OPHI) \WPHI P \Big)
		P(\partial_\nu\OPHI) \FE^\lambda \WPHI P
\\
	& \qquad
			+ \Big( P\OPHI \WPHI P \Big) P (\partial_\nu\OPHI) \FE^\lambda (\partial_\mu \WPHI) P
		= \Big( \partial_\mu (\WPHI P\OPHI)_k \Big) P(\partial_\nu\OPHI) \FE^\lambda\FE^k \WPHI P.
\end{aligned}
\label{E:IDTWO}
\end{equation}
Taking the difference of equations \eqref{E:IDONE} and \eqref{E:IDTWO}, we obtain
\begin{equation}
\begin{aligned}
	& \frac{1}{2} \Big( P\OPHI (\partial_\mu\WPHI)P - P(\partial_\mu \OPHI) \WPHI P \Big)
		\, \partial_\nu (P\OPHI \FE^\lambda \WPHI P)
\\
	& \quad
	+ P\OPHI \WPHI P \Big( P (\partial_\mu\OPHI) \FE^\lambda (\partial_\nu \WPHI) P
		- P (\partial_\nu\OPHI) \FE^\lambda (\partial_\mu \WPHI) P \Big)
\\
	& \quad\qquad
	= \Big( \partial_\mu (\WPHI P\OPHI)_k \Big) \Big( P\OPHI \FE^k\FE^\lambda (\partial_\nu\WPHI) P
		- P(\partial_\nu\OPHI) \FE^\lambda\FE^k \WPHI P \Big).
\end{aligned}
\label{E:ALMOST}
\end{equation}

\medskip

\textbf{Step~2.} For the second term on the left-hand side of \eqref{E:ALMOST}, we have that
\begin{align*}
	& P (\partial_\mu\OPHI) \FE^\lambda (\partial_\nu \WPHI) P
		- P (\partial_\nu\OPHI) \FE^\lambda (\partial_\mu \WPHI) P
\\
	& \qquad
		= \partial_\mu \Big( P \OPHI \FE^\lambda (\partial_\nu \WPHI) P
	    	- P (\partial_\nu\OPHI) \FE^\lambda \WPHI P \Big)
		- \Big( P \OPHI \FE^\lambda (\partial^2_{\mu\nu} \WPHI) P
			- P (\partial^2_{\mu\nu}\OPHI) \FE^\lambda \WPHI P \Big).
\end{align*}
The second term on the right-hand side, which we will denote by $S^\lambda_{\mu\nu}$ in the following, is symmetric in $\mu,\nu$ because mixed second derivatives commute.

We restrict $\lambda$ to indices $l = 1 \ldots 3$ and use \eqref{E:ANTICO2} to obtain
\begin{align*}
	& P\OPHI \FE^k \FE^l (\partial_\nu\WPHI) P
		- P(\partial_\nu\OPHI) \FE^l \FE^k \WPHI P
\\
	& \qquad
	= \delta^{kl} \Big( P\OPHI (\partial_\nu\WPHI) P - P(\partial_\nu\OPHI) \WPHI P \Big)
		- i{\varepsilon^{kl}}_o \, \partial_\nu (P\OPHI \FE^o \WPHI P).
\end{align*}
Collecting terms, we arrive at the identity
\[
\begin{aligned}
	& \frac{1}{2} \Big( P\OPHI (\partial_\mu\WPHI)P - P(\partial_\mu \OPHI) \WPHI P \Big)
		\, \partial_\nu (P\OPHI \FE^l \WPHI P)
\\
	& \quad
	+ P\OPHI \WPHI P \bigg( \partial_\mu \Big( P \OPHI \FE^l (\partial_\nu \WPHI) P
	    	- P (\partial_\nu\OPHI) \FE^l \WPHI P \Big)
		- S^l_{\mu\nu} \bigg)
\\
	& \quad\qquad
	= \Big( \partial_\mu (\WPHI P\OPHI)_k \Big)
		\bigg( \delta^{kl} \Big( P\OPHI (\partial_\nu\WPHI) P - P(\partial_\nu\OPHI) \WPHI P \Big)
			- i{\varepsilon^{kl}}_o \, \partial_\nu (P\OPHI \FE^o \WPHI P) \bigg).
\end{aligned}
\]

\medskip

\textbf{Step~3.} We identify terms with hydrodynamic quantities. First, we have
\[
	P\OPHI \FE^\lambda \WPHI P
		= 2 (P\OPHI \FE^\lambda \WPHI P)_0 P
		= 2 (\WPHI P \OPHI)_\lambda P
		= 2 (\phi \overline{P} \phi^\dagger)^\lambda P
		= 2 j^\lambda P
\]
for $\lambda = 0\ldots 3$, using $P^2=P$, \eqref{E:EXPA2}, and \eqref{E:LIFT}. Second, we observe that
\begin{align*}
	& P\OPHI \FE^\lambda (\partial_\nu\WPHI) P - P (\partial_\nu\OPHI) \FE^\lambda \WPHI P
		 = 2 \Big( P\OPHI \FE^\lambda (\partial_\nu\WPHI) P - P (\partial_\nu\OPHI) \FE^\lambda \WPHI P \Big)_0 P
\\
	& \qquad
		= 2 \bigg( \FE^\lambda \Big( (\partial_\nu\WPHI) P \OPHI - \WPHI P (\partial_\nu\OPHI) \Big) \bigg)_0 P
		= 4i ({T_\nu}^\lambda -\Gamma_\nu j^\lambda) P
\end{align*}
for $\nu,\lambda = 0\ldots 3$; see again \eqref{E:EXPA2}, \eqref{E:ROTAT}, and \eqref{E:NULL}. It follows that
\begin{equation}
\begin{aligned}
	& \frac{1}{2} \, 4i ({T_\mu}^0 -\Gamma_\mu j^0) \, \partial_\nu (2j^l)
		+ 2 j^0 \, \Big( 4i \, \partial_\mu({T_\nu}^l-\Gamma_\nu j^l) - S^l_{\mu\nu} \Big)
\\
	& \qquad
		= \sum_k \partial_\mu j^k \bigg( \delta^{kl} \, 4i ({T_\nu}^0-\Gamma_\nu j^0)
			- i{{\varepsilon}^{kl}}_o \, \partial_\nu (2j^o) \bigg),
\end{aligned}
\label{E:READY}
\end{equation}
after cancelling the factor $P$. Simplifying and rearranging terms, we arrive at
\[
\begin{aligned}
	& \Big( {T_\mu}^0 (\partial_\nu j^l) + 2 j^0 \partial_\mu {T_\nu}^l \Big)
		- \Big( \Gamma_\mu j^0 (\partial_\nu j^l) + 2 j^0 \partial_\mu(\Gamma_\nu j^l) \Big)
		- \frac{1}{2i} j^0 S^l_{\mu\nu}
\\
	& \qquad
		= \sum_k \partial_\mu j^k \bigg( \delta^{kl} ({T_\nu}^0-\Gamma_\nu j^0)
			-\frac{1}{2} {{\varepsilon}^{kl}}_o (\partial_\nu j^o) \bigg).
\end{aligned}
\]

\medskip

\textbf{Step~4.} In the following, we will use the notation introduced in \eqref{E:NOTAT}.

We now restrict $\mu,\nu$ to $m,n = 1\ldots 3$, multiply \eqref{E:READY} by $\varepsilon^{imn} v_l$, and sum over $m,n,l = 1\ldots 3$. Then $S_{mn}$ drops out, by symmetry. We further observe that
\[
	\varepsilon^{imn} \Big( \partial_m j^l ({T_n}^0-\Gamma_n j^0) \Big)
		= -\varepsilon^{imn} ({T_m}^0 -\Gamma_m j^0) \partial_n j^l,
\]
which we move to the left-hand side of \eqref{E:READY}. We have that
\[
	v_l (\partial_n j^l) = (v_l v^l) \partial_n \RHO + \RHO v_l (\partial_n v^l)
		= -\partial_n \RHO
\]
because $v^l = -v_l$ and $v_l v^l = -1$, thus $v_l(\partial_nv^l) = 0$. Similarly, we find
\begin{equation}
	v_l \Big( \Gamma_m (\partial_n j^l) + \partial_m (\Gamma_n j^l) \Big)
		= -\Big( \Gamma_m (\partial_n\RHO) + \Gamma_n (\partial_m\RHO) \Big)
			-\RHO (\partial_m \Gamma_n).
\label{E:MORE}
\end{equation}
Since the first term on the right-hand side of \eqref{E:MORE} is symmetric in $m,n$, it drops out when multiplied against $\varepsilon^{imn}$ and summed over $m,n$. We have
\begin{equation}
\begin{aligned}
	& v_l \Big( {T_m}^0 (\partial_n j^l) + \RHO (\partial_m {T_n}^l) \Big)
\\
	& \qquad
		= v_l \Bigg( p_m (\partial_n j^l) + \RHO \partial_m \bigg( p_n v^l
			- \frac{1}{2} \RHO {\varepsilon^{lk}}_o v_k (\partial_n v^o) \bigg) \Bigg)
\\
	& \qquad
		= -\RHO \Big( u_m (\partial_n\RHO) + u_n (\partial_m\RHO) \Big)
			- \RHO^2 (\partial_m u_n)
		- \frac{1}{2} \RHO^2 {\varepsilon^{lk}}_o v_l (\partial_m v_k) (\partial_n v^o),
\end{aligned}
\label{E:LOLI}
\end{equation}
writing $p_m = \RHO u_m$ and using Lemma~\ref{L:EXPR}. Here we have used that
\[
	v_l \partial_m \Big( \RHO {\varepsilon^{lk}}_o v_k (\partial_n v^o) \Big)
		= \RHO {\varepsilon^{lk}}_o v_l (\partial_m v_k) (\partial_n v^o)
\]
because ${\varepsilon^{kl}}_o v_k v_l = 0$, by symmetry. Again the first term on the right-hand side of \eqref{E:LOLI} drops out after multiplication by $\varepsilon^{imn}$ and summation in $m,n$. Finally
\begin{align*}
	\sum_k {\varepsilon^{kl}}_o v_l (\partial_m j^k) (\partial_n j^o)
		= \sum_k \RHO^2 {\varepsilon^{kl}}_o v_l (\partial_m v^k) (\partial_n v^o)
\end{align*}
because $\sum_k {\varepsilon^{kl}}_o v_l v^k = 0$ etc., by symmetry and \eqref{E:LC}. It follows that
\begin{align*}
	& 2 \RHO^2 \varepsilon^{imn} \bigg( \partial_m (-u_n + \Gamma_n)
		- \frac{1}{2} {\varepsilon^{lk}}_o v_l (\partial_m v_k) (\partial_n v^o) \bigg)
\\
	& \qquad
		= -\frac{1}{2} \RHO^2 \varepsilon^{imn} \sum_k {\varepsilon^{kl}}_o v_l (\partial_m v^k) (\partial_n v^o).
\end{align*}
Using Remark~\ref{R:SFG}, we observe that
\begin{gather*}
	{\varepsilon^{lk}}_o v_l (\partial_m v_k)
		= \varepsilon_{lko} v^l (\partial_m v^k),
\\
	-\sum_k {\varepsilon^{kl}}_o v_l (\partial_m v^k)
		= \varepsilon_{klo} v^l (\partial_m v^k)
		= -\varepsilon_{lko} v^l (\partial_m v^k).
\end{gather*}
Cancelling and rearranging terms, we then obtain \eqref{E:QUCO}.
\end{proof}


\printbibliography

@article{BarutZanghi1984,
  author  = {Barut, A. O. and Zangh\`{i}, N.},
  title   = {Classical model of the {D}irac electron},
  journal = {Phys. Rev. Lett.},
  volume  = {52},
  year    = {1984},
  pages   = {2009--2012},
  doi     = {10.1103/PhysRevLett.52.2009}
}

@incollection{Baylis1996,
  author    = {Baylis, W. E.},
  title     = {The paravector model of spacetime},
  booktitle = {Clifford (geometric) algebras ({B}anff, {AB}, 1995)},
  pages     = {237--252},
  publisher = {Birkh\"{a}user, Boston, MA},
  year      = {1996}
}

@article{BialynickiBirula1995,
  author  = {Bia\l ynicki-Birula, I.},
  title   = {Hydrodynamic form of the {W}eyl equation},
  journal = {Acta Phys. Polon. B},
  volume  = {26},
  year    = {1995},
  pages   = {1201--1208}
}

@book{CazenaveHarauxMartel1998,
  author    = {Cazenave, T. and Haraux, A. and Martel, Y.},
  title     = {An Introduction to Semilinear Evolution Equations},
  publisher = {Oxford University Press},
  year      = {1998},
  doi       = {10.1093/oso/9780198502777.001.0001}
}

@book{DaviauBertrandSocrounGirardot2022,
  author    = {Daviau, C. and Bertrand, J. and Socroun, T. and Girardot, D.},
  title     = {Developing the Theory of Everything},
  year      = {2022},
  publisher = {Fondation Louis de Broglie},
  isbn      = {978-2-910458005},
  url       = {https://www.researchgate.net/publication/362092441_Developing_the_Theory_of_Everything}
}

@article{Hestenes1973,
  author  = {Hestenes, D.},
  title   = {Local observables in the Dirac theory},
  journal = {J. Math. Phys.},
  volume  = {14},
  year    = {1973},
  pages   = {893--905}
}

@article{Hestenes2020a,
  author  = {Hestenes, D.},
  title   = {Quantum Mechanics of the electron particle-clock},
  journal = {arXiv preprint arXiv:1910.10478},
  year    = {2020}
}

@article{Hestenes2020b,
  author  = {Hestenes, D.},
  title   = {Zitterbewegung structure in electrons and photons},
  journal = {arXiv preprint arXiv:1910.11085},
  year    = {2020}
}

@article{Ivanenko1938,
  author  = {Ivanenko, D.},
  title   = {Notes to the theory of interaction via particles},
  journal = {Zh. Eksp. Teor. Fiz},
  volume  = {8},
  year    = {1938},
  pages   = {260--266}
}

@article{Madelung1926,
  author  = {Madelung, E.},
  title   = {Eine anschauliche Deutung der Gleichung von Schrödinger},
  journal = {Naturwissenschaften},
  volume  = {14},
  year    = {1926},
  pages   = {1004},
  doi     = {10.1007/BF01504657}
}

@article{Madelung1927,
  author  = {Madelung, E.},
  title   = {Quantentheorie in hydrodynamischer Form},
  journal = {Z. Phys.},
  volume  = {40},
  year    = {1927},
  pages   = {322--326},
  doi     = {10.1007/BF01400372}
}

@article{Soler1970,
  title   = {Classical, Stable, Nonlinear Spinor Field with Positive Rest Energy},
  author  = {Soler, M.},
  journal = {Phys. Rev. D},
  volume  = {1},
  pages   = {2766--2769},
  year    = {1970},
  doi     = {10.1103/PhysRevD.1.2766}
}

@article{Takabayasi1957,
  title   = {Relativistic Hydrodynamics of the Dirac Matter. Part I. General Theory},
  author  = {Takabayasi, T.},
  journal = {Progress of Theoretical Physics Supplement},
  year    = {1957},
  volume  = {4},
  pages   = {1-80}
}

@book{Tretter2008,
  author    = {Tretter, C.},
  title     = {Spectral Theory of Block Operator Matrices and Applications},
  publisher = {World Scientific},
  year      = {2008},
  doi       = {10.1142/p493}
}


\end{document}